\PassOptionsToPackage{unicode}{hyperref}
\PassOptionsToPackage{naturalnames}{hyperref}
\documentclass[12pt]{amsart}
\usepackage[margin=1in]{geometry}
\usepackage[backend=biber]{biblatex}
\usepackage[breaklinks=true]{hyperref}
\usepackage{graphicx, amsmath, amsthm, amssymb, tikz, xcolor, blkarray, booktabs,bigstrut,setspace,nicematrix}

\usetikzlibrary{matrix, fit, positioning, calc, decorations.pathreplacing}
\newtheorem{thm}{Theorem}[section]
\newtheorem*{thm*}{Theorem}
\newtheorem{lem}[thm]{Lemma}
\newtheorem{cor}[thm]{Corollary}
\newtheorem{prp}[thm]{Proposition}

\theoremstyle{definition}
\newtheorem{dfn}[thm]{Definition}

\theoremstyle{remark}
\newtheorem{rmk}[thm]{Remark}
\newtheorem{ex}[thm]{Example}
\newtheorem{qstn}[thm]{Open Problem}

\newcommand{\wt}{\mathrm{wt}}

\newlength\runit
\edef\Radius#1{#1\runit}

\newcommand*\circled[1]{\tikz[baseline=(char.base)]{\node[shape=circle,draw,inner sep=2pt] (char) {#1};}}

\newcommand{\blackdot}{
    \begin{tikzpicture}[baseline,x=.75cm, y=.5cm]
    \draw[circle, fill=black](0, 0) circle[radius = 1mm] node {};
    \end{tikzpicture}
}

\newcommand{\whitedot}{
    \begin{tikzpicture}[baseline,x=.75cm, y=.5cm]
    \draw[circle, fill=white](0, 0) circle[radius = 1mm] node {};
    \end{tikzpicture}
}

\newcommand{\topgraydot}{
    \begin{tikzpicture}[baseline,x=.75cm, y=.5cm]
    \draw[circle, fill=gray](0, 0) circle[radius = 1mm] node (D1) {};
    \draw[dotted,thick] (-.25,-.4) -- (-.25,.4) {};
    \draw[dotted,thick] (-.25,.4) -- (.25,.4) {};
    \draw[dotted,thick] (.25,.4) -- (.25, -.4) {};
    \end{tikzpicture}
}

\newcommand{\bottomgraydot}{
    \begin{tikzpicture}[baseline,x=.75cm, y=.5cm]
    \draw[circle, fill=gray](0, 0) circle[radius = 1mm] node (D1) {};
    \draw[dotted,thick] (-.25,-.4) -- (-.25,.4) {};
    \draw[dotted,thick] (.25,.4) -- (.25, -.4) {};
    \draw[dotted,thick] (.25, -.4) -- (-.25,-.4) {};
    \end{tikzpicture}
}

\author{Robert Angarone}
\email{angar017@umn.edu}
\title{Cylindrical Networks and Total Nonnegativity}

\begin{document}

\begin{abstract}
We prove that an infinite block-Toeplitz matrix with finite diagonal support is totally nonnegative if and only if it is the weight matrix of a cylindrical network. This generalizes a well-known theorem of Brenti concerning finite totally nonnegative matrices and planar networks; in particular, our work gives an alternative, self-contained proof of the non-square case. Our argument employs Temperley-Lieb immanants, first introduced by Rhoades and Skandera, which are certain elements of Lusztig's dual canonical bases. As an application, we also obtain a new proof of a well-known theorem relating totally nonnegative block-Toeplitz matrices to interlacing polynomials.
\end{abstract}

\maketitle

\section{Introduction}

A real matrix is called \textit{totally nonnegative} (TN) if all of its minors, i.e. the determinants of its square submatrices, are nonnegative. A \textit{Toeplitz} matrix is one which is constant along diagonals, i.e. entry $(i,j)$ is always equal to entry $(i+1,j+1)$. Accordingly, an infinite block-Toeplitz matrx is one of this form:
\[\begin{bmatrix}
A_0 & A_1 & A_2 & A_3 & \ldots & A_i & A_{i+1} & A_{i+2} \ldots \\
0 & A_0 & A_1 & A_2 & \ldots & A_{i-1} & A_i & A_{i+1} \ldots \\
0 & 0 & A_0 & A_1 & \ldots & A_{i-2} & A_{i-1} & A_i & \ldots \\
\vdots & \vdots & \vdots & \vdots & \vdots & \vdots & \vdots & \ddots
\end{bmatrix} = T\left(\underline{A}\right),\]
where $\underline{A}=(A_0, A_1, \ldots)$ is a sequence of real matrices which all have the same dimensions. If they are all $n \times m$ matrices, we may also call such a matrix \textit{$(n,m)$-periodic}, since in this case the block-Toeplitz condition is equivalent to requiring that entry $(i,j)$ must always be equal to entry $(i+n, j+m)$. We say that such a matrix has \textit{finite diagonal support} if the sequence $\underline{A}$ has only finitely many nonzero terms.

A \textit{cylindrical network} is a planar, edge-weighted, acyclic digraph embedded on a cylinder, with $n$ labelled sources on one boundary component and $m$ labelled sinks on the other.
We associate to each cylindrical network $N$ an infinite block-Toeplitz matrix called its \textit{weight matrix} $W(N)$. For $i\leq n, \, j \leq m,$ and $w, w' \in \mathbb{N}$, the $(i + wn, \, j + w'm)$ entry of $W(N)$ is given by $\sum_p \wt(p)$, where the sum ranges over all directed paths $p$ in $N$ from source $i$ to sink $j$ that wrap around the cylinder $w'-w$ times, and $\wt(p)$ is the product of all the edges in $p$. The primary contribution of this paper is the following result:
\begin{thm}\label{thm:main}
Any infinite, $(n,m)$-periodic TN block-Toeplitz matrix with finite diagonal support is equal to the weight matrix of some cylindrical network with $n$ sources, $m$ sinks, and nonnegative edge weights.
\end{thm}
When considered in tandem with results of Lam and Pylyavskyy, our main theorem completes a cylindrical, infinite generalization of a well-known story concerning \textit{planar} networks and \textit{finite} TN matrices. Brenti proved that any finite TN matrix is equal to the weight matrix of a finite planar network with nonnegative edge weights \cite{brenti-tp}. Conversely, the famed Gessel-Lindstr\"{o}m-Viennot lemma implies that the weight matrix of any finite planar network with nonnegative edge weights is TN \cite{lindstrom,gessel-viennot}.

Our Theorem~\ref{thm:main} is the cylindrical generalization of Brenti's theorem, while Lam and Pylyavskyy introduced the cylindrical generalziation of the Gessel-Lindstr\"{o}m-Viennot lemma \cite{lam-pylyavskyy}. Lam and Pylyavskyy phrase their results in terms of the \textit{totally nonnegative loop group}, a construction equivalent to invertible $(n,n)$-periodic matrices.

In Sections~\ref{sec:background_tn}~through~\ref{sec:background_loop}, we place the main theorem in context and recall some important known results. In Sections~\ref{sec:background_tl}~and~\ref{sec:cmd}, we recall key lemmas concerning Temperley-Lieb immanants, which are used to great extent in the proof of the main theorem. Section~\ref{sec:main} is dedicated to proving the result; in particular, Section~\ref{sec:redux} explains how the main result reduces to Proposition~\ref{lem:main}, and Section~\ref{sec:proof-of-main-lemma} is dedicated to proving that proposition. In Section~\ref{sec:interlacing}, we discuss applications of our technique to interlacing polynomials and pose an open question in this vein.

\section{Background}

\subsection{Totally Nonnegative Matrices}\label{sec:background_tn}

Although TN matrices were first studied in analysis, they have since become of great interest to combinatorialists. This is due to their appearance in the theory of real-rooted and log-concave polynomials \cite{brenti-polya}, their prominent role in laying the groundwork for the theory of cluster algebras \cite{FZ-tests-params, FZ-clusters-1}, and more. We refer the reader to \cite{brenti-tp,FZ-tests-params} for more comprehensive overviews of total positivity.

If $I$ is a subset of the rows of $M$ and $J$ is a subset of its columns, we let $M[I,J]$ denote the submatrix of $M$ with row set $I$ and column set $J$. We also allow for $I$ and $J$ to be multisets of the rows and columns of $M$; in this way, we form \textit{generalized} submatrices of $M$. Note that if $M$ is a TN matrix, then any generalized submatrix of $M$ is also TN.

The following theorem, known as the Loewner-Whitney theorem, gives simple generators for the set of invertible TN matrices.

\begin{thm}[\cite{Loewner, whitney}]\label{thm:semigroup-generators}
As a semigroup, the set of invertible TN matrices is generated by matrices of the following forms, for nonnegative $a$, $b$, and $c$:
\[\begin{bmatrix}
1 \\
& 1 \\
& & \ddots \\
& & & a \\
& & & & \ddots \\
& & & & & 1
\end{bmatrix},
\begin{bmatrix}
1 \\
& 1 \\
& & \ddots \\
& & & 1 & b \\
& & & & \ddots \\
& & & & & 1
\end{bmatrix},
\begin{bmatrix}
1 \\
& 1 \\
& & \ddots \\
& & & 1 \\
& & & c & \ddots \\
& & & & & 1
\end{bmatrix}.\]
\end{thm}

\subsection{The Gessel-Lindstr\"{o}m-Viennot Lemma}\label{sec:background_glv}

The following fundamental result, due to Lindstr\"{o}m and later Gessel-Viennot, establishes a combinatorial method for generating TN matrices. It is a cornerstone of algebraic combinatorics, and is used, for example, to provide a modern proof of the Jacobi-Trudi identity.

\begin{dfn}\label{dfn:weight-matrix}
Suppose $N$ is a finite, planar, acyclic, edge-weighted digraph with $n$ sources and $m$ sinks (i.e. a \textit{planar network}). Then the weight matrix $W(N)$ is given by \[W(N)_{i,j} = \sum_{\substack{\text{paths } p \\ \text{ source }i \to \text{ sink } j}} \mathrm{wt}(p), \qquad \mathrm{wt}(p) := \prod_{e \in p} \mathrm{wt}(e).\]
\end{dfn}

\begin{lem}[\cite{lindstrom,gessel-viennot}]\label{lem:classical-GLV}
Suppose $N$ is a planar network as in Definition~\ref{dfn:weight-matrix}. Given subsets $I = \{i_1<i_2<\ldots<i_k\} \subseteq [n]$ and  $J = \{j_1<j_2<\ldots<j_k\} \subseteq [m]$, we have \[ \left| W(N)[I,J] \right| = \sum_{\substack{\text{noncrossing path families } \\ p_{1}, p_{2}, \ldots, p_{k} \\ p_\ell : \text{ source } i_\ell \to \text{ sink } j_\ell}} \prod_{\ell=1}^k \mathrm{wt}(p_{k}).\] In particular, if $N$ has nonnegative edge weights, then $W(N)$ is totally nonnegative.
\end{lem}

What is especially surprising is that the converse of Lemma~\ref{lem:classical-GLV} is true; that is, given any totally nonnegative matrix $M$, there exists a planar network $N$ with nonnegative edge weights such that $W(N)=M$. If $M$ is invertible, this can be proven by showing that each of the matrices in Theorem~\ref{thm:semigroup-generators} can be represented as a network. One then observes that concatenating networks, sink-to-source, corresponds to matrix multiplicaiton. If $M$ is not invertible, or perhaps not even square, it can still be represented by a network, but one must use a more complicated argument. This was first observed by Brenti \cite{brenti-tp}.

\begin{thm}[\cite{brenti-tp}]\label{thm:TN-representable}
Any TN matrix is the weight matrix of a planar network with nonnegative edge weights.
\end{thm}

Brenti's argument relies on a certain tridiagonal factorization theorem, ultimately coming from theorems in \cite{ando} and \cite{cryer}. Our proof of the cylindrical version of this theorem will also involve a factorization result; Brenti's methods, however, do not immediately apply to our case. While we could apply existing results involving finite TN matrices locally on any square submatrix of a given infinite TN matrix, the difficulty is in showing that we can find a network (equivalently, factorization) structure which is compatible globally with the entire matrix.

\subsection{Total Nonnegativity in the Loop Group}\label{sec:background_loop}

Lam and Pylyavskyy \cite{lam-pylyavskyy} introduced an extension of the notion of total nonnegativity to matrices whose entries are polynomials (or power series) with real coefficients, rather than real numbers. These turn out to be a special case of $(n,m)$-periodic matrices called the \textit{loop group}. In this section, we recall their core results concerning the relationship between the totally nonnegative loop group and cylindrical networks.

\begin{dfn}
The \textit{$n \times n$ polynomial loop group} is the set of $n \times n$ matrices with Laurent polynomial entries whose determinants are nonzero monomials. An element of the loop group is called a \textit{loop}.
\end{dfn}

\begin{dfn}\label{dfn:unfolding}
Suppose $M$ is an $n \times m$ matrix with formal Laurent series entries. Let $M_d$ be the matrix whose $i,j$th entry is the coefficient of $t^d$ in $M_{i,j}$. Then define the \textit{unfolding} of $M$ to be the infinite block-Toeplitz matrix with copies of $M_i$ along the $i$th block-diagonal. We say $M$ is TN if its unfolding is TN.
\end{dfn}

\begin{dfn}\label{dfn:folding}
Define the \textit{folding} of an infinite $(n,m)$-periodic matrix $M$ to be the unique $n \times m$ matrix with Laurent series entries whose unfolding is $M$.
\end{dfn}

\begin{rmk}\label{rmk:folding-is-homo}
Unfolding and folding preserve multiplication and addition of matrices.
\end{rmk}

\begin{ex}\label{eg:unfolding}
Here is a member of the $2 \times 2$ loop group and its associated unfolding:
\begin{align*}
{\renewcommand*{\arraystretch}{2.3}
\left[\begin{array}{cc}
\dfrac{1}{1-6t} & \dfrac{2}{1-6t}  \\
\dfrac{3t}{1-6t} & \dfrac{1}{1-6t}  \\
\end{array} \right]}
&=
\begin{pmatrix}
1 & 2 \\
0 & 1
\end{pmatrix}
+
\begin{pmatrix}
6 & 12 \\
3 & 6
\end{pmatrix} t
+
\begin{pmatrix}
36 & 72 \\
18 & 36
\end{pmatrix} t^2
+ \ldots \\
& \underset{\text{Unfold}}{\leadsto}
\left[\begin{array}{cc|cc|cc|c}
1 & 2 & 6 & 12 & 36 & 72 & \ldots \\
0 & 1 & 3 & 6 & 18 & 36 & \ldots \\
\hline
0 & 0 & 1 & 2 & 6 & 12 & \ldots \\
0 & 0 & 0 & 1 & 3 & 6 & \ldots \\
\hline
\vdots & \vdots & \vdots & \vdots & \vdots & \vdots & \ddots
\end{array} \right]
\end{align*}
\end{ex}

\begin{dfn}
A \textit{cylindrical network} is finite, planar, edge-weighted, acyclic digraph embedded on a cylinder, with $n$ sources $v_1, v_2, \ldots, v_n$ on one boundary component and $m$ sinks $w_1, w_2, \ldots, w_m$ on the other, along with a distinguished line $\mathfrak{h}$ connecting the boundary components of the cylinder. By default, $\mathfrak{h}$ originates between sources $v_1$ and $v_n$, terminates between sinks $w_1$ and $w_m$, and does not intersect any vertex of $N$.
\end{dfn}

\begin{dfn}
Given a path $p$ in a cylindrical network $N$ with line $\mathfrak{h}$, define $\mathrm{rot}(p)$ to be the \textit{rotor} of $p$: the number of times $p$ crosses $\mathfrak{h}$ counterclockwise minus the number of times $p$ crosses $\mathfrak{h}$ clockwise.
\end{dfn}

\begin{dfn}
Given $i,j \in \mathbb{N}$, let $\overline{i} \in \{1,2,\ldots,n\}$ be equal to $i \bmod n$; similarly for $\overline{j}$ and $m$. Then a path $p$ in a cylindrical network $N$ is said to be an \textit{$(i,j)$-path} if it begins at source $v_{\overline{i}}$, ends at sink $w_{\overline{j}}$, and has $\mathrm{rot}(p) = \left(\left(j-\overline{j}\right)/m\right) - \left(\left(i - \overline{i}\right)/n\right)$.
\end{dfn}

\begin{rmk}
We will always draw cylindrical networks with sources on the left, sinks on the right, and the line $\mathfrak{h}$ on top. A crossing of $\mathfrak{h}$ is `counterclockwise' if it would be counterclockwise when viewed from the right of the cylinder; in other words, it is counterclockwise if it moves from the back of our view in the diagrams to the front; see Example~\ref{eg:cylindrical-network}

Intuitively, when forming an $(i,j)$-path, one should think that crossing the line $\mathfrak{h}$ counterclockwise has the effect of increasing all sink labels by $m$, as though we crossed into a new copy of the cylinder. Accordingly, crossing $\mathfrak{h}$ clockwise has the effect of decreasing path labels by $m$. 
\end{rmk}

\begin{dfn}
Given a cylindrical network $N$, the \textit{folded weight matrix} of $N$ is the $n \times m$ matrix given by
    \[\overline{W}(N)_{i,j} = \sum_{\substack{\text{paths } p \\ v_i \to w_j}} t^{\mathrm{rot}(p)}\mathrm{wt}(p).\]
The \textit{unfolded weight matrix} is the infinite matrix given by 
    \[W(N)_{i,j} = \sum_{(i,j)-\text{paths } p} \mathrm{wt}(p).\]
Note that the unfolded weight matrix is $(n,m)$-periodic, and indeed is the unfolding of the folded weight matrix.
\end{dfn}

\begin{dfn}
Suppose $p$ is an $(i,j)$-path and $q$ is an $(i',j')$-path in a cylindrical network $N$ such that the two paths intersect at vertex $c$. Then let $\tilde{p}$ and $\tilde{q}$ be the paths obtained by following either $p$ or $q$ until the point $c$, then following the other path afterwards. We say that $p$ and $q$ are \textit{properly crossing} if $\tilde{p}$ is an $(i,j')$-path (equivalently, if $\tilde{q}$ is an $(i,j')$-path); the paths are called \textit{uncrossed} otherwise.
\end{dfn}

\begin{ex}\label{eg:cylindrical-network}
Below is a cylindrical network with three source vertices, three sink vertices, and one additional vertex. All edges are assumed to have weight $1$. The line $\mathfrak{h}$ coincides with the top of the picture. Dotted portions of edges are assumed to lie on the back side of the cylinder. The red path $p$ from $v_1$ to $w_3$ is a $(1,6)$-path. The green path $r$ from $v_2$ to $w_2$ is a $(2,2)$-path. The blue path $q$ from $v_3$ to $w_1$ is a $(3,4)$-path. The paths $p$ and $q$ are uncrossed, even though they intersect.
\begin{center}
\includegraphics[width=.5\linewidth]{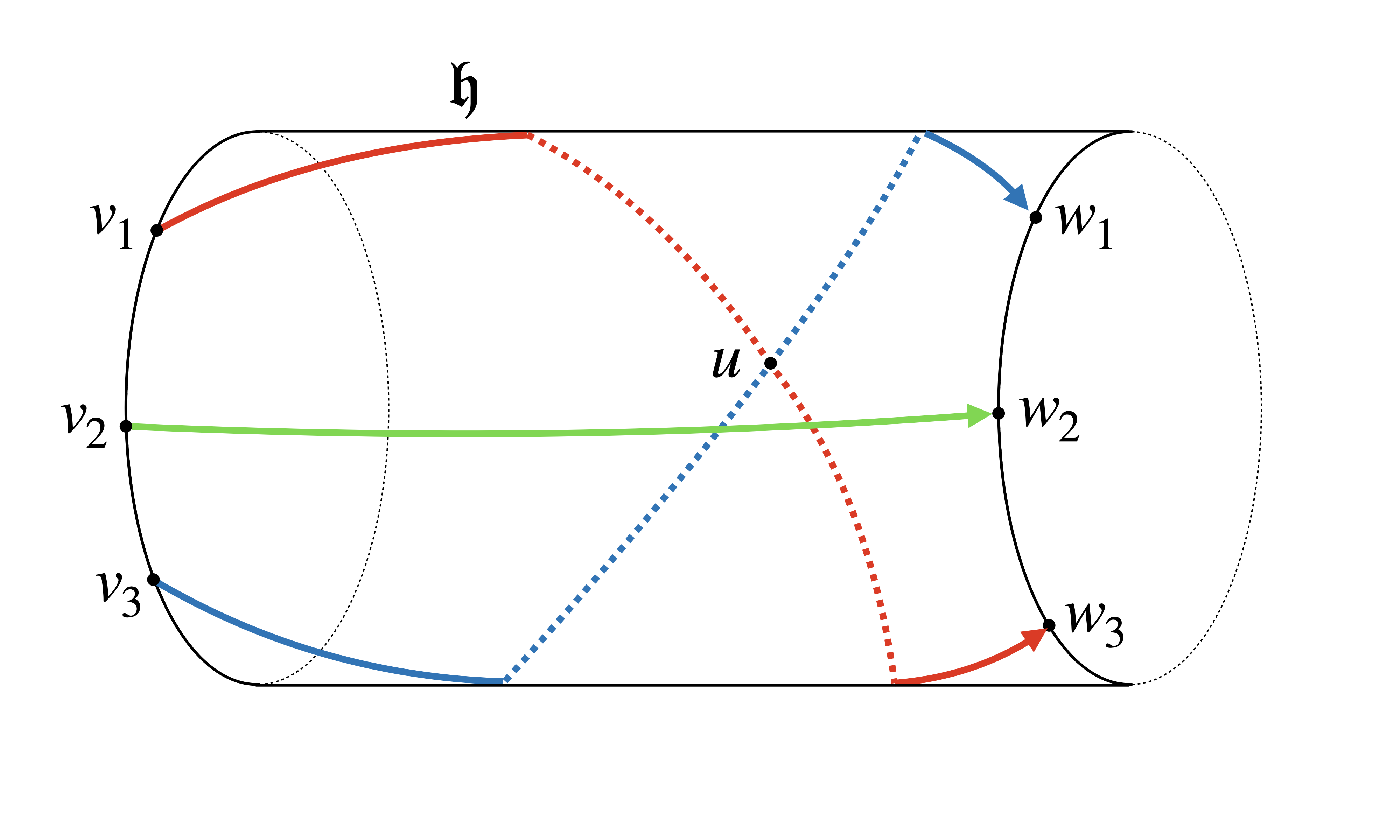}
\end{center}
The folded and unfolded weight matrices for this network are as follows:
\begin{align*}
\begin{bmatrix}
1 & 0 & t^{-1} \\
0 & 1 & 0 \\
t & 0 & 1 \\
\end{bmatrix}
& \underset{\text{Unfold}}{\leadsto}
\left[\begin{array}{ccc|ccc|c}
1 & 0 & 0 & 0 & 0 & 0 & \ldots \\
0 & 1 & 0 & 0 & 0 & 0 & \ldots \\
0 & 0 & 1 & 1 & 0 & 0 & \ldots \\ \hline
0 & 0 & 1 & 1 & 0 & 0 & \ldots \\
0 & 0 & 0 & 0 & 1 & 0 & \ldots \\
0 & 0 & 0 & 0 & 0 & 1 & \ldots \\ \hline
\vdots & \vdots & \vdots & \vdots & \vdots & \vdots & \ddots
\end{array} \right]
\end{align*}
\end{ex}

Lam and Pylyavskyy showed that the above definitions give rise to a cylindrical analogue of the classical Gessel-Lindstr\"{o}m-Viennot lemma. 

\begin{lem}[\cite{lam-pylyavskyy}]\label{lem:cylindrical-GLV}
Suppose $N$ is a cylindrical network, and $W(N)$ its weight matrix. For $I = \{i_1<i_2<\ldots<i_k\} \subseteq \mathbb{N}$ and  $J = \{j_1<j_2<\ldots<j_k\} \subseteq \mathbb{N}$, let $\Phi(I,J)$ be the set of all families of pairwise uncrossed paths $\{p_1, \ldots, p_k\}$ such that $p_\ell$ is an $(i_\ell, j_\ell)$-path for all $1 \leq \ell \leq k$. Then \[|W(N)[I,J]| = \sum_{P \in \Phi(I,J)} \prod_{p \in P} \mathrm{wt}(p).\] In particular, if $N$ has nonnegative edge weights, then $W(N)$ is totally nonnegative.
\end{lem}

In fact, by proving factorization theorems for the loop group, Lam and Pylyavskyy were also able to prove the following partial generalization of Brenti's theorem.

\begin{thm}[\cite{lam-pylyavskyy}]\label{thm:invertible-loops-representable}
Any $(n,n)$-periodic, invertible block-Toeplitz TN matrix is the weight matrix of a cylindrical network with nonnegative edge weights.
\end{thm}

Our primary goal in the remainder of the paper is to extend this generalization to include $(n,m)$-periodic matrices for which $n \neq m$ and which are not necessarily invertible. The proofs methods of Lam and Pylyavskyy will not readily extend to this context. Although we will also proceed by factorization argument, those authors relied on a lemma which applies only to full-rank matrices---namely that, for such matrices, checking that only row- or column-solid minors are positive is enough for total positivity. Instead, we will employ an argument which checks all minors explicitly, and which employs certain combinatorial matrix functions which we describe next.

\subsection{Temperley-Lieb Immanants}\label{sec:background_tl}

In this section, we recall some facts about Temperley-Lieb immanants, first introduced by Rhoades and Skandera~\cite{rhoades-skandera}. Temperley-Lieb immanants are certain functions on matrices that are defined via the combinatorics of noncrossing matchings. They are generally unwieldy to compute. Luckily, we need not compute even a single Temperley-Lieb immanant explicitly in order to prove the main theorem. Instead, we just rely on a pair of key results from Rhoades and Skandera, which appear in this work as Lemmas~\ref{lem:tl-pos}~and~\ref{lem:tl-compat}. The next two sections are dedicated to stating these lemmas.

\begin{dfn}\label{dfn:tl-alg}
Given \(n \geq 2\) and a formal parameter \(\xi\), the \textit{Temperley-Lieb algebra} \(\mathrm{TL}_n(\xi)\) is a \(\mathbb{C}[\xi]\)-algebra generated by \(t_1, t_2, \ldots, t_{n-1}\) and subject to the following relations:
    \[t_i^2 = \xi t_i, \quad t_it_jt_i = t_i \text{ if } |i-j|=1, \,\, \text{ and } \,\, t_it_j = t_jt_i \text{ if } |i-j| \geq 2.\]
\end{dfn}

\begin{rmk}\label{rmk:tl-visual}
The multiplication rules in $\mathrm{TL}_n(\xi)$ can be interpreted visually. To each $t_i$, associate a noncrossing matching of $2n$ points in the following way:
    \begin{center}
    \begin{tikzpicture}[x=.75cm, y=.5cm]
    \draw[circle, fill=black](0, 0) circle[radius = \Radius{1}] node {};
    \draw[circle, fill=black](0, -1) circle[radius = \Radius{1}] node {};
    \draw[circle, fill=black](0, -2) circle[radius = \Radius{1}] node {};
    \draw[circle, fill=black](0, -4) circle[radius = \Radius{1}] node {};
    \draw[circle, fill=black](0, -5) circle[radius = \Radius{1}] node {};
    \draw[circle, fill=black](1, -5) circle[radius = \Radius{1}] node {};
    \draw[circle, fill=black](1, -4) circle[radius = \Radius{1}] node {};
    \draw[circle, fill=black](1, -2) circle[radius = \Radius{1}] node {};
    \draw[circle, fill=black](1, -1) circle[radius = \Radius{1}] node {};
    \draw[circle, fill=black](1, 0) circle[radius = \Radius{1}] node {};
    \draw (0, 0) to [bend left=45] (0, -1);
    \draw (1, 0) to [bend right=45] (1, -1);
    \draw (0, -2) to (1, -2);
    \draw (0, -4) to (1, -4);
    \draw (0, -5) to (1, -5);
    \draw (.5, -3) node {\(\vdots\)};
    \draw (1.5,-4.5) node[right] {,};
    \draw (.5,-6) node[below] {\(t_1\)};
    \end{tikzpicture}
    \begin{tikzpicture}[x=.75cm, y=.5cm]
    \draw[circle, fill=black](0, 0) circle[radius = \Radius{1}] node {};
    \draw[circle, fill=black](0, -1) circle[radius = \Radius{1}] node {};
    \draw[circle, fill=black](0, -2) circle[radius = \Radius{1}] node {};
    \draw[circle, fill=black](0, -4) circle[radius = \Radius{1}] node {};
    \draw[circle, fill=black](0, -5) circle[radius = \Radius{1}] node {};
    \draw[circle, fill=black](1, -5) circle[radius = \Radius{1}] node {};
    \draw[circle, fill=black](1, -4) circle[radius = \Radius{1}] node {};
    \draw[circle, fill=black](1, -2) circle[radius = \Radius{1}] node {};
    \draw[circle, fill=black](1, -1) circle[radius = \Radius{1}] node {};
    \draw[circle, fill=black](1, 0) circle[radius = \Radius{1}] node {};
    \draw (0, 0) to (1, 0);
    \draw (0, -1) to [bend left=45] (0, -2);
    \draw (1, -1) to [bend right=45] (1, -2);
    \draw (0, -4) to (1, -4);
    \draw (0, -5) to (1, -5);
    \draw (.5, -3) node {\(\vdots\)};
    \draw (1.5,-4.5) node[right] {\(, \ldots,\)};
    \draw (.5,-6) node[below] {\(t_2\)};
    \end{tikzpicture}
    \begin{tikzpicture}[x=.75cm, y=.5cm]
    \draw[circle, fill=black](0, 0) circle[radius = \Radius{1}] node {};
    \draw[circle, fill=black](0, -1) circle[radius = \Radius{1}] node {};
    \draw[circle, fill=black](0, -2) circle[radius = \Radius{1}] node {};
    \draw[circle, fill=black](0, -4) circle[radius = \Radius{1}] node {};
    \draw[circle, fill=black](0, -5) circle[radius = \Radius{1}] node {};
    \draw[circle, fill=black](1, -5) circle[radius = \Radius{1}] node {};
    \draw[circle, fill=black](1, -4) circle[radius = \Radius{1}] node {};
    \draw[circle, fill=black](1, -2) circle[radius = \Radius{1}] node {};
    \draw[circle, fill=black](1, -1) circle[radius = \Radius{1}] node {};
    \draw[circle, fill=black](1, 0) circle[radius = \Radius{1}] node {};
    \draw (0, 0) to (1, 0);
    \draw (0, -1) to (1, -1);
    \draw (0, -2) to (1, -2);
    \draw (.5, -3) node {\(\vdots\)};
    \draw (0, -4) to [bend left=45] (0, -5);
    \draw (1, -4) to [bend right=45] (1, -5);
    \draw (1.5,-4.5) node[right] {.};
    \draw (.5,-6) node[below] {\(t_{n-1}\)};
    \end{tikzpicture}
    \end{center}
To multiply matchings, first concatenate them left-to-right; second, remove all loops to obtain a single noncrossing matching of \(2n\) points; and finally, multiply the resulting diagram by a factor of \(\xi\) for each loop removed.  
\end{rmk}

\begin{ex}
In \(\mathrm{TL}_5(\xi)\), we have
        \begin{align*}
    (t_1t_2)(t_2+t_3) &=
    t_1t_2^2 + t_1t_2t_3 \\
    &=
    \begin{tikzpicture}[baseline,x=.75cm, y=.5cm]
    \begin{scope}[shift={(0,2.5)}]
    \draw[circle, fill=black](0, 0) circle[radius = 1mm] node {};
    \draw[circle, fill=black](0, -1) circle[radius = 1mm] node {};
    \draw[circle, fill=black](0, -2) circle[radius = 1mm] node {};
    \draw[circle, fill=black](0, -3) circle[radius = 1mm] node {};
    \draw[circle, fill=black](0, -4) circle[radius = 1mm] node {};
    \draw[circle, fill=black](1, -4) circle[radius = 1mm] node {};
    \draw[circle, fill=black](1, -3) circle[radius = 1mm] node {};
    \draw[circle, fill=black](1, -2) circle[radius = 1mm] node {};
    \draw[circle, fill=black](1, -1) circle[radius = 1mm] node {};
    \draw[circle, fill=black](1, 0) circle[radius = 1mm] node {};
    \draw (0, 0) to [bend left=45] (0, -1);
    \draw (1, 0) to [bend right=45] (1, -1);
    \draw (0, -2) to (1, -2);
    \draw (0, -3) to (1, -3);
    \draw (0, -4) to (1, -4);
    \end{scope}
    \begin{scope}[shift={(1,2.5)}]
    \draw[circle, fill=black](0, 0) circle[radius = 1mm] node {};
    \draw[circle, fill=black](0, -1) circle[radius = 1mm] node {};
    \draw[circle, fill=black](0, -2) circle[radius = 1mm] node {};
    \draw[circle, fill=black](0, -3) circle[radius = 1mm] node {};
    \draw[circle, fill=black](0, -4) circle[radius = 1mm] node {};
    \draw[circle, fill=black](1, -4) circle[radius = 1mm] node {};
    \draw[circle, fill=black](1, -3) circle[radius = 1mm] node {};
    \draw[circle, fill=black](1, -2) circle[radius = 1mm] node {};
    \draw[circle, fill=black](1, -1) circle[radius = 1mm] node {};
    \draw[circle, fill=black](1, 0) circle[radius = 1mm] node {};
    \draw (0, 0) to (1, 0);
    \draw (0, -1) to [bend left=45] (0, -2);
    \draw (1, -1) to [bend right=45] (1, -2);
    \draw (0, -3) to (1, -3);
    \draw (0, -4) to (1, -4);
    \end{scope}
    \begin{scope}[shift={(2,2.5)}]
    \draw[circle, fill=black](0, 0) circle[radius = 1mm] node {};
    \draw[circle, fill=black](0, -1) circle[radius = 1mm] node {};
    \draw[circle, fill=black](0, -2) circle[radius = 1mm] node {};
    \draw[circle, fill=black](0, -3) circle[radius = 1mm] node {};
    \draw[circle, fill=black](0, -4) circle[radius = 1mm] node {};
    \draw[circle, fill=black](1, -4) circle[radius = 1mm] node {};
    \draw[circle, fill=black](1, -3) circle[radius = 1mm] node {};
    \draw[circle, fill=black](1, -2) circle[radius = 1mm] node {};
    \draw[circle, fill=black](1, -1) circle[radius = 1mm] node {};
    \draw[circle, fill=black](1, 0) circle[radius = 1mm] node {};
    \draw (0, 0) to (1, 0);
    \draw (0, -1) to [bend left=45] (0, -2);
    \draw (1, -1) to [bend right=45] (1, -2);
    \draw (0, -3) to (1, -3);
    \draw (0, -4) to (1, -4);
    \end{scope}
    \end{tikzpicture}
    +
    \begin{tikzpicture}[baseline,x=.75cm, y=.5cm]
    \begin{scope}[shift={(-1,2.5)}]
    \draw[circle, fill=black](0, 0) circle[radius = 1mm] node {};
    \draw[circle, fill=black](0, -1) circle[radius = 1mm] node {};
    \draw[circle, fill=black](0, -2) circle[radius = 1mm] node {};
    \draw[circle, fill=black](0, -3) circle[radius = 1mm] node {};
    \draw[circle, fill=black](0, -4) circle[radius = 1mm] node {};
    \draw[circle, fill=black](1, -4) circle[radius = 1mm] node {};
    \draw[circle, fill=black](1, -3) circle[radius = 1mm] node {};
    \draw[circle, fill=black](1, -2) circle[radius = 1mm] node {};
    \draw[circle, fill=black](1, -1) circle[radius = 1mm] node {};
    \draw[circle, fill=black](1, 0) circle[radius = 1mm] node {};
    \draw (0, 0) to [bend left=45] (0, -1);
    \draw (1, 0) to [bend right=45] (1, -1);
    \draw (0, -2) to (1, -2);
    \draw (0, -3) to (1, -3);
    \draw (0, -4) to (1, -4);
    \end{scope}
    \begin{scope}[shift={(0,2.5)}]
    \draw[circle, fill=black](0, 0) circle[radius = 1mm] node {};
    \draw[circle, fill=black](0, -1) circle[radius = 1mm] node {};
    \draw[circle, fill=black](0, -2) circle[radius = 1mm] node {};
    \draw[circle, fill=black](0, -3) circle[radius = 1mm] node {};
    \draw[circle, fill=black](0, -4) circle[radius = 1mm] node {};
    \draw[circle, fill=black](1, -4) circle[radius = 1mm] node {};
    \draw[circle, fill=black](1, -3) circle[radius = 1mm] node {};
    \draw[circle, fill=black](1, -2) circle[radius = 1mm] node {};
    \draw[circle, fill=black](1, -1) circle[radius = 1mm] node {};
    \draw[circle, fill=black](1, 0) circle[radius = 1mm] node {};
    \draw (0, 0) to (1, 0);
    \draw (0, -1) to [bend left=45] (0, -2);
    \draw (1, -1) to [bend right=45] (1, -2);
    \draw (0, -3) to (1, -3);
    \draw (0, -4) to (1, -4);
    \end{scope}
    \begin{scope}[shift={(1,2.5)}]
    \draw[circle, fill=black](0, 0) circle[radius = 1mm] node {};
    \draw[circle, fill=black](0, -1) circle[radius = 1mm] node {};
    \draw[circle, fill=black](0, -2) circle[radius = 1mm] node {};
    \draw[circle, fill=black](0, -3) circle[radius = 1mm] node {};
    \draw[circle, fill=black](0, -4) circle[radius = 1mm] node {};
    \draw[circle, fill=black](1, -4) circle[radius = 1mm] node {};
    \draw[circle, fill=black](1, -3) circle[radius = 1mm] node {};
    \draw[circle, fill=black](1, -2) circle[radius = 1mm] node {};
    \draw[circle, fill=black](1, -1) circle[radius = 1mm] node {};
    \draw[circle, fill=black](1, 0) circle[radius = 1mm] node {};
    \draw (0, 0) to (1, 0);
    \draw (0, -1) to (1, -1);
    \draw (0, -2) to [bend left=45] (0, -3);
    \draw (1, -2) to [bend right=45] (1, -3);
    \draw (1,-4) to (0,-4);
    \end{scope}
    \end{tikzpicture} \\
    &=
    \xi \cdot \begin{tikzpicture}[baseline,x=.75cm, y=.5cm]
    \begin{scope}[shift={(0,2.5)}]
    \draw[circle, fill=black](0, 0) circle[radius = 1mm] node {};
    \draw[circle, fill=black](0, -1) circle[radius = 1mm] node {};
    \draw[circle, fill=black](0, -2) circle[radius = 1mm] node {};
    \draw[circle, fill=black](0, -3) circle[radius = 1mm] node {};
    \draw[circle, fill=black](0, -4) circle[radius = 1mm] node {};
    \draw[circle, fill=black](1, -4) circle[radius = 1mm] node {};
    \draw[circle, fill=black](1, -3) circle[radius = 1mm] node {};
    \draw[circle, fill=black](1, -2) circle[radius = 1mm] node {};
    \draw[circle, fill=black](1, -1) circle[radius = 1mm] node {};
    \draw[circle, fill=black](1, 0) circle[radius = 1mm] node {};
    \draw (0, 0) to [bend left=45] (0, -1);
    \draw (0, -2) to (1, 0);
    \draw (1, -1) to [bend right=45] (1, -2);
    \draw (0, -3) to (1, -3);
    \draw (0, -4) to (1, -4);
    \end{scope}
    \end{tikzpicture}
    +
    \begin{tikzpicture}[baseline,x=.75cm, y=.5cm]
    \begin{scope}[shift={(0,2.5)}]
    \draw[circle, fill=black](0, 0) circle[radius = 1mm] node {};
    \draw[circle, fill=black](0, -1) circle[radius = 1mm] node {};
    \draw[circle, fill=black](0, -2) circle[radius = 1mm] node {};
    \draw[circle, fill=black](0, -3) circle[radius = 1mm] node {};
    \draw[circle, fill=black](0, -4) circle[radius = 1mm] node {};
    \draw[circle, fill=black](1, -4) circle[radius = 1mm] node {};
    \draw[circle, fill=black](1, -3) circle[radius = 1mm] node {};
    \draw[circle, fill=black](1, -2) circle[radius = 1mm] node {};
    \draw[circle, fill=black](1, -1) circle[radius = 1mm] node {};
    \draw[circle, fill=black](1, 0) circle[radius = 1mm] node {};
    \draw (0, 0) to [bend left=45] (0, -1);
    \draw (0, -2) to (1, 0);
    \draw (0, -3) to (1, -1);
    \draw (1, -2) to [bend right=45] (1, -3);
    \draw (0, -4) to (1, -4);
    \end{scope}
    \end{tikzpicture}.
    \end{align*}
\end{ex}

As is more apparent from the visual interpretation of multiplication, the set of \textit{all} possible noncrossing matchings of \(2n\) points forms a basis for \(\mathrm{TL}_n(\xi)\). Thus, as a \(\mathbb{C}[\xi]\)-vector space, the dimension of \(\mathrm{TL}_n(\xi)\) is the \(n\)th Catalan number $C_n$. For example, since $C_3=5$, we know $\mathrm{TL}_3(\xi)$ is $5$-dimensional and has the following basis:
    \[\mathcal{B}_3 = \left\{\vcenter{\hbox{\begin{tikzpicture}[baseline,x=.75cm, y=.5cm]
    \draw[circle, fill=black](0, 0) circle[radius = 1mm] node {};
    \draw[circle, fill=black](0, -1) circle[radius = 1mm] node {};
    \draw[circle, fill=black](0, -2) circle[radius = 1mm] node {};
    \draw[circle, fill=black](1, -2) circle[radius = 1mm] node {};
    \draw[circle, fill=black](1, -1) circle[radius = 1mm] node {};
    \draw[circle, fill=black](1, 0) circle[radius = 1mm] node {};
    \draw (0, -2) to (1, -2);
    \draw (0, 0) to (1, 0);
    \draw (0, -1) to (1, -1);
    \end{tikzpicture}}}, \,\,
    \vcenter{\hbox{\begin{tikzpicture}[baseline,x=.75cm, y=.5cm]
    \draw[circle, fill=black](0, 0) circle[radius = 1mm] node {};
    \draw[circle, fill=black](0, -1) circle[radius = 1mm] node {};
    \draw[circle, fill=black](0, -2) circle[radius = 1mm] node {};
    \draw[circle, fill=black](1, -2) circle[radius = 1mm] node {};
    \draw[circle, fill=black](1, -1) circle[radius = 1mm] node {};
    \draw[circle, fill=black](1, 0) circle[radius = 1mm] node {};
    \draw (0, -2) to (1, -2);
    \draw (0, 0) to [bend left=45] (0, -1);
    \draw (1, 0) to [bend right=45] (1, -1);
    \end{tikzpicture}}}, \,\,
    \vcenter{\hbox{\begin{tikzpicture}[baseline,x=.75cm, y=.5cm]
    \draw[circle, fill=black](0, 0) circle[radius = 1mm] node {};
    \draw[circle, fill=black](0, -1) circle[radius = 1mm] node {};
    \draw[circle, fill=black](0, -2) circle[radius = 1mm] node {};
    \draw[circle, fill=black](1, -2) circle[radius = 1mm] node {};
    \draw[circle, fill=black](1, -1) circle[radius = 1mm] node {};
    \draw[circle, fill=black](1, 0) circle[radius = 1mm] node {};
    \draw (0, 0) to (1, 0);
    \draw (0, -1) to [bend left=45] (0, -2);
    \draw (1, -1) to [bend right=45] (1, -2);
    \end{tikzpicture}}}, \,\,
    \vcenter{\hbox{\begin{tikzpicture}[baseline,x=.75cm, y=.5cm]
    \draw[circle, fill=black](0, 0) circle[radius = 1mm] node {};
    \draw[circle, fill=black](0, -1) circle[radius = 1mm] node {};
    \draw[circle, fill=black](0, -2) circle[radius = 1mm] node {};
    \draw[circle, fill=black](1, -2) circle[radius = 1mm] node {};
    \draw[circle, fill=black](1, -1) circle[radius = 1mm] node {};
    \draw[circle, fill=black](1, 0) circle[radius = 1mm] node {};
    \draw (0, 0) to (1, -2);
    \draw (1, 0) to [bend right=45] (1, -1);
    \draw (0, -1) to [bend left=45] (0, -2);
    \end{tikzpicture}}}, \,\,
    \vcenter{\hbox{\begin{tikzpicture}[baseline,x=.75cm, y=.5cm]
    \draw[circle, fill=black](0, 0) circle[radius = 1mm] node {};
    \draw[circle, fill=black](0, -1) circle[radius = 1mm] node {};
    \draw[circle, fill=black](0, -2) circle[radius = 1mm] node {};
    \draw[circle, fill=black](1, -2) circle[radius = 1mm] node {};
    \draw[circle, fill=black](1, -1) circle[radius = 1mm] node {};
    \draw[circle, fill=black](1, 0) circle[radius = 1mm] node {};
    \draw (0, -2) to (1, 0);
    \draw (0, 0) to [bend left=45] (0, -1);
    \draw (1, -1) to [bend right=45] (1, -2);
    \end{tikzpicture}}}\right\}.\]

\begin{dfn}\label{dfn:tl-basis}
Let $\mathcal{B}_n$ be the set of noncrossing matchings on $2n$ points, regarded as a basis for $\mathrm{TL}_n(\xi)$.
\end{dfn}

\begin{dfn}\label{dfn:tl-phi}
Given $T \in \mathcal{B}_n$, let \(\varphi_T: S_n \to \mathbb{R}\) be given as follows: if \(w \in S_{n}\) has reduced word \(s_{j_1}s_{j_2} \cdots s_{j_k}\), let \(\varphi_T(w)\) be the coefficient of \(T\) when \(\left(t_{j_1}-1\right)\left(t_{j_2}-1\right) \cdots \left(t_{j_k}-1\right)\) is expanded in the Temperley-Lieb algebra $\mathrm{TL}_n(2)$.
\end{dfn}

\begin{dfn}\label{dfn:imm}
An \textit{immanant} is any function $I: \mathrm{Mat}_{n \times n}(\mathbb{R}) \to \mathbb{R}$ of the form
    \[I(M) = \sum_{w \in S_n} \varphi(w) M_{1, w(1)} M_{2, w(2)} \cdots M_{n, w(n)}\]
for some function $\varphi:S_n \to \mathbb{R}$, where $S_n$ is the symmetric group.
\end{dfn}

\begin{dfn}\label{dfn:tl-imm}
Given a noncrossing matching $T \in \mathcal{B}_n$, the \textit{Temperley-Lieb immanant} for \(T\) is the function $\mathrm{imm}_T \colon \mathrm{Mat}_n\left(\mathbb{R}\right) \to \mathbb{R}$ defined by
    \[\mathrm{imm}_T(M) := \sum_{w \in S_n} \varphi_T(w) M_{1,w(1)} M_{1,w(2)} \cdots M_{n,w(n)}.\]
\end{dfn}

\begin{rmk}\label{rmk:tl-notation}
When the matrix $M$ is understood, we will frequently draw the noncrossing matching $T$ rather than write $\mathrm{imm}_T(M)$.
\end{rmk}

\begin{lem}[Rhoades-Skandera \cite{rhoades-skandera}]\label{lem:tl-pos}
Suppose $M$ is a totally nonnegative $n \times n$ matrix. Then for any $T \in \mathcal{B}_n$, we have $\mathrm{imm}_T(M) \geq 0$.
\end{lem}

The positivity property in Lemma~\ref{lem:tl-pos} is remarkable. It remains an intriguing open problem to classify all functions which are nonnegative on totally nonnegative matrices; these are known as \textit{totally positive functions}. It is known, for example, that immanants obtained by setting $\varphi$ equal to an irreducible $S_n$ character are totally positive~\cite{stembridge}.

\subsection{Complementary Minor Immanants and Their Diagrams}\label{sec:cmd}

In this section, we describe notation for complementary minor immanants, i.e. pairs of minors with complementary row and column sets. This notation was introduced by Rhoades and Skandera \cite{rhoades-skandera} to state Lemma~\ref{lem:tl-compat}, which will be indispensable in the proof of Theorem~\ref{thm:main}.

\begin{dfn}\label{dfn:comp-imms}
Given an $n \times n$ matrix $M$ and $I,J \subseteq [n]$ with $|I|=|J|$, the \textit{complementary minor immanant} of $M$ with respect to $I$ and $J$ is
    \[\mathcal{C}_{I,J}(M) := \left|M[I,J]\right|\cdot\left|M[[n] \setminus I, \,\, [n] \setminus J]\right|.\]
\end{dfn}

\begin{dfn}\label{dfn:comp-minor-diagram}
Given $I,J \subseteq [n]$ with $|I|=|J|$, the associated \textit{complementary minor diagram} is a $2 \times n$ array of dots colored black or white in the following way:
    \begin{itemize}
    \item in the left column, color the \(j\)th dot from the top black if \(j \in J\) and white otherwise;
    \item in the right column, color the \(i\)th dot from the top white if \(i \in I\) and black otherwise.
    \end{itemize}
\end{dfn}

\begin{ex}\label{ex:cmd}
If $n=5$, $I=\{1,2,4\}$, and $J=\{1,3,5\}$, Definition~\ref{dfn:comp-minor-diagram} gives the following diagram:
    \[\begin{tikzpicture}[baseline,x=.75cm, y=.5cm]
    \begin{scope}[shift={(0,2.1)}]
    \draw[circle, fill=black](0, 0) circle[radius = 1mm] node {};
    \draw[circle, fill=white](0, -1) circle[radius = 1mm] node {};
    \draw[circle, fill=black](0, -2) circle[radius = 1mm] node {};
    \draw[circle, fill=white](0, -3) circle[radius = 1mm] node {};
    \draw[circle, fill=black](0, -4) circle[radius = 1mm] node {};
    \draw[circle, fill=black](1, -4) circle[radius = 1mm] node {};
    \draw[circle, fill=white](1, -3) circle[radius = 1mm] node {};
    \draw[circle, fill=black](1, -2) circle[radius = 1mm] node {};
    \draw[circle, fill=white](1, -1) circle[radius = 1mm] node {};
    \draw[circle, fill=white](1, 0) circle[radius = 1mm] node {};
    \end{scope}
    \end{tikzpicture}.\]
\end{ex}

\begin{rmk}\label{rmk:cm-notation}
Since $I$ and $J$ can be read off of a complementary minor diagram, when the matrix $M$ is understood, we will frequently draw a complementary minor diagram rather than writing $\mathcal{C}_{I,J}(M)$.
\end{rmk}

\begin{dfn}\label{dfn:tl-compat}
Given $I,J \subseteq [n]$ with $|I|=|J|$, let $\Theta(I,J)$ be the set of noncrossing matchings $T \in \mathcal{B}_n$ such that every strand in $T$ connects a white dot to a black dot in the complementary minor diagram associated to $I$ and $J$. If $T \in \Theta(I,J)$, we say $T$ is $(I,J)$-compatible.
\end{dfn}

\begin{lem}[Rhoades-Skandera \cite{rhoades-skandera}]\label{lem:tl-compat}
For any $n \times n$ matrix $M$ and $I, J \subseteq [n]$ with $|I|=|J|$, we have
    \[\mathcal{C}_{I,J}(M) = \sum_{T \in \Theta(I,J)} \mathrm{imm}_T(M).\]
\end{lem}

\begin{ex}\label{ex:compatibility}
Applying Lemma~\ref{lem:tl-compat} to the complementary minor diagram from Example~\ref{ex:cmd}, we obtain
    \input{Tikz-Pictures/compatibility-example}
We are suppressing the matrix $M$ in both sides of the equation; see Remarks~\ref{rmk:tl-notation}~and~\ref{rmk:cm-notation}.
\end{ex}

\begin{cor}\label{lem:tl-zero}
If $M$ is a totally nonnegative matrix and $T \in \mathcal{B}_n$ is $(I,J)$-compatible for some pair of sets $I,J$ such that $\mathcal{C}_{I,J}(M)=0,$ then $\mathrm{imm}_T(M)=0$.
\end{cor}

\begin{proof}
By Lemma~\ref{lem:tl-compat}, the immanant $\mathcal{C}_{I,J}$ is a positive linear combination of Temperley-Lieb immanants, one of which is $\mathrm{imm}_T(M)$. By Lemma~\ref{lem:tl-pos}, each of those summands is nonnegative. Since their sum is $\mathcal{C}_{I,J}(M)=0$, each Temperley-Lieb immanant in that sum must be $0$ individually. In particular, $\mathrm{imm}_T(M)=0$.
\end{proof}

\section{Main Result}\label{sec:main}

\subsection{Southwest Corner Elimination}\label{sec:redux}

In this section, we explain how Theorem~\ref{thm:main} reduces to Proposition~\ref{lem:main}, which states that certain row operations on infinite block-Toeplitz matrices preserve total nonnegativity. The main idea is that any given $(n,m)$-periodic TN matrix $M$ can be written as $M = R \cdot M'$, where $R$ is a simple matrix achieving a periodic row operation and $M'$ is an $(n,m)$-periodic matrix with fewer nonzero entries. The matrix $R$ will easily be seen to be representable as a cylindrical network; that $M'$ is representable will follow by induction. Before presenting Proposition~\ref{lem:main}, we first provide some prerequisite definitions. Then we will prove Theorem~\ref{thm:main} on the assumption of Proposition~\ref{lem:main}.


\begin{dfn}\label{dfn:periodic-row-ops}
We say that an $(n,m)$-periodic matrix $M'$ is obtained from another such matrix $M$ via a \textit{periodic elementary row operation} if
    \[M'_{i,j} = \begin{cases}
    M_{i,j} - c \cdot M_{i-1, j} & \text{ if } i \equiv k \bmod n \\
    M_{i,j} & \text{ otherwise}
    \end{cases}\]
For some real constant $c$ and residue class $k$ modulo $n$.
\end{dfn}

\begin{lem}\label{lem:elimination-representable}
Suppose $M = R \cdot M'$, where $M'$ is obtained from $M$ via a periodic elementary row operation. Then the matrix $R$ is equal to the weight matrix of a cylindrical network. 
\end{lem}

\begin{proof}
Observe that $R$ is an $(n,n)$-periodic matrix with $1$s on the main diagonal, the constant $c$ in all entries of the form $(k, k-1)$ modulo $n$, and $0$s elsewhere. This is the weight matrix of the cylindrical network with $n$ edges of weight $1$ from each source $v_i$ to the corresponding sink $w_i$, and one edge of weight $c$ from source $v_k$ to sink $w_{k-1}$ (taking indices modulo $n$). For example, with $n=7$ and $k=3$, we obtain the following network:
    \[\includegraphics[scale=.06]{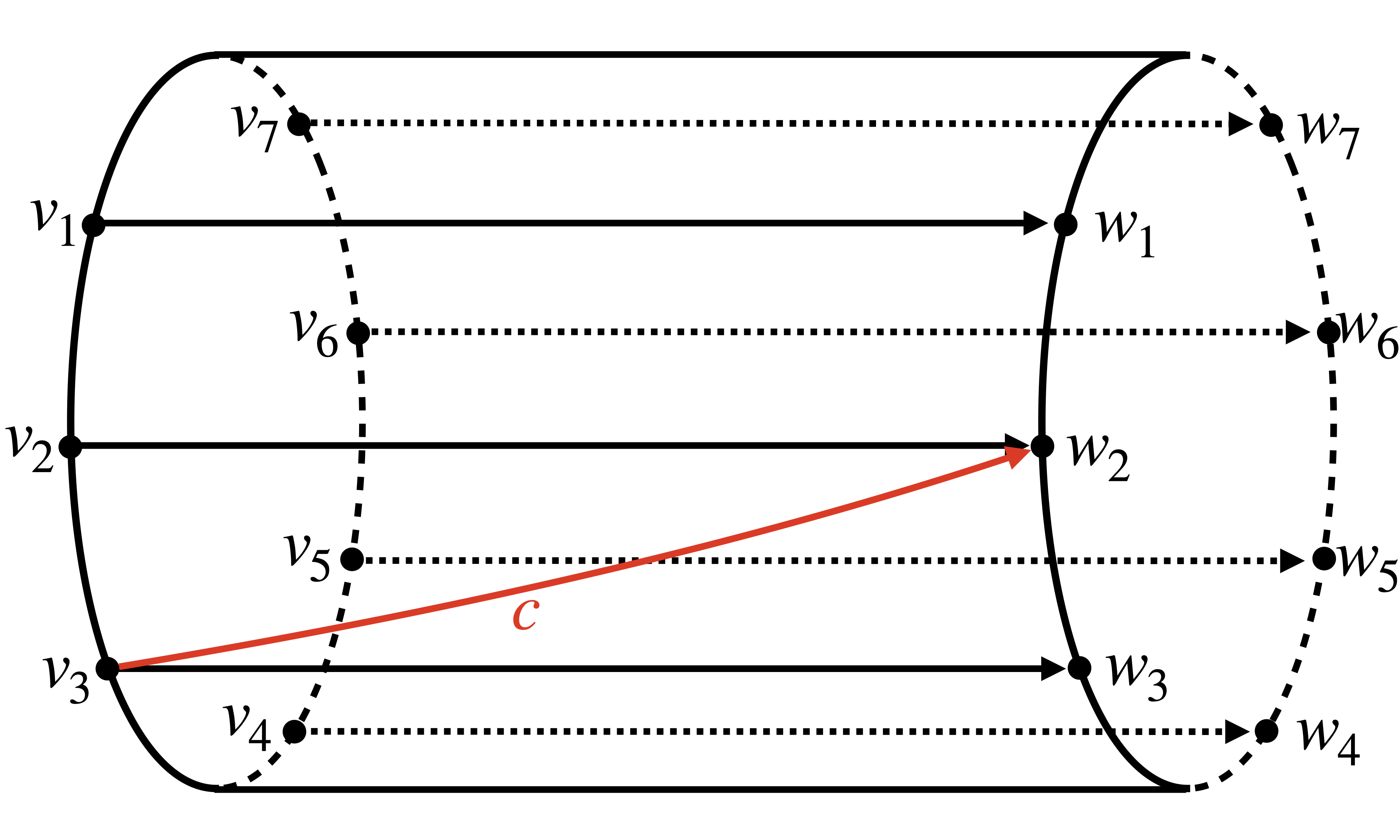}. \qedhere\]
\end{proof}

\begin{dfn}\label{dfn:convex}
A matrix $M$ is \textit{row-convex} if there does not exist a row index $\ell$ and column indices $i < j < k$ so that $M_{\ell, i} \neq 0$, $M_{\ell, j}=0$, and $M_{\ell,k} \neq 0$; in other words, each row of the matrix is a sequence with no internal zeroes. Define \textit{column-convex} matrices similarly, and say a matrix has \textit{convex support} if it is both row- and column-convex.
\end{dfn}

\begin{prp}\label{lem:tn-is-convex}
Any TN matrix with no row or column of only zeroes has convex support.
\end{prp}

\begin{proof}
Suppose $M$ is such a matrix and suppose, for the sake of contradiction, that it does not have convex support. Then it is either not row-convex or not column-convex. In the former case, we may find a row index $\ell$ and column indices $i < j < k$ so that $M_{\ell, i} \neq 0$, $M_{\ell, j}=0$, and $M_{\ell,k} \neq 0$. Since the matrix has no columns of all zeroes, there exists a nonzero entry in column $j$. If $M_{\ell', j} \neq 0$ for $\ell' < \ell$, then the following minor of $M$ is negative, a contradiction:
    \[\left|\begin{matrix} M_{\ell', i} & M_{\ell', j} \\ M_{\ell, i} & M_{\ell, j}\end{matrix}\right| = M_{\ell', i}M_{\ell, j}-M_{\ell', j}M_{\ell, i} = M_{\ell', i}\cdot 0-M_{\ell', j}M_{\ell, i} = -M_{\ell', j}M_{\ell, i}.\]
Similarly, if $\ell'>\ell$, we find the following negative minor:
    \[\left|\begin{matrix} M_{\ell, j} & M_{\ell, k} \\ M_{\ell',j} & M_{\ell',k} \end{matrix}\right| = -M_{\ell,k}M_{\ell',j}.\]
A similar argument can be used to show column-convexity.
\end{proof}

\begin{dfn}\label{dfn:SW-corners}
A \textit{SW corner} of a matrix is a nonzero entry with all entries below and to the left of it equal to $0$. A \textit{special SW corner} is a SW corner $M_{i,j}$ which is also the only nonzero entry below or to the left of $M_{i-1,j}$. Visually, a special SW corner looks like
\[\begin{tikzpicture}
\draw (0,0) grid (5,4);
\draw [fill=blue!10!white] (0,0) grid (5,4) rectangle (0,0);
\draw [fill=black!10!white] (0,0) grid (3,3) rectangle (0,0);
\draw [fill=blue!60!white] (2,1) grid (3,3) rectangle (2,1);
\node at (2.5, 1.5) {\Large $\star$};
\end{tikzpicture}\]
where $\star$ represents the special SW corner itself. Entries in solid blue are required to be nonzero, entries in gray are required to be $0$, and entries in light blue could be zero or nonzero.
\end{dfn}

\begin{lem}\label{lem:SW-corner-existence}
Suppose $M$ is a TN $(n,m)$-periodic infinite matrix with finite diagonal support, no rows of all zeroes, and $n > m$. Then $M$ has a special SW corner.
\end{lem}

\begin{proof}
Let $\ell(i) = \min \{j : M_{i,j} \neq 0\}$ be the column index of the leftmost nonzero entry in row $i$; this is well-defined since we assumed $M$ has finite diagonal support. By Proposition~\ref{dfn:periodic-row-ops}, $M$ is column-convex, so the sequence $\ell(1), \ell(2), \ldots$ is either weakly increasing or decreasing. Since $M$ is $(n,m)$-periodic, we have $\ell(i+n) = \ell(i)+m$, so the sequence must in fact be weakly increasing.

Observe that $\ell(1), \ell(2), \ldots, \ell(n+1)$ is a weakly increasing sequence with $\ell(n+1) - \ell(1) = m$. In other words, it is an increasing sequence of length $n+1$ that takes at most $m+1 < n+1$ values. Thus two consecutive elements in this sequence must be equal; suppose $\ell(i^\star-1)$ and $\ell(i^\star)$ are the last such pair and let $j^\star = \ell(i^\star)$.

We claim that $M_{i^\star,j^\star}$ is a special SW corner of $M$. By construction, it is the leftmost nonzero entry in its row, and the same is true for $M_{i^\star-1, j^\star}$. Also by construction, $M_{i^\star+1, j^\star} = 0$; if this entry were nonzero, we would have $\ell(i^\star+1) \leq \ell(i^\star)$, which implies $\ell(i^\star+1)=\ell(i^\star)$ since this sequence is weakly increasing. But we selected $i^\star$ to be the last possible choice of $i \in \{1,\ldots,n+1\}$ with $\ell(i-1)=\ell(i)$, so this is a contradiction. Given that $M_{i^\star+1, j^\star} = 0$, all other entries below or to the left of $M_{i^\star,j^\star}$ are $0$ as well, since $M$ has convex support. Thus $M_{i^\star,j^\star}$ satisfies all the conditions to be a special southwest corner of $M$.
\end{proof}

\begin{prp}\label{lem:main}
Suppose $M$ is a TN $(n,m)$-periodic infinite matrix with finite diagonal support and a special SW corner. Then the matrix $M'$ obtained from $M$ by eliminating a special SW corner of $M$ via a periodic elementary row operation is also TN.
\end{prp}

\begin{proof}[Proof of Theorem~\ref{thm:main}, assumsing Proposition~\ref{lem:main}]
Let $M$ be an arbitrary TN infinite $(n,m)$-periodic matrix with finite diagonal support. We show, via an argument by cases, that we can always write $M=R \cdot M'$ or $M=M' \cdot R$ so that:
\begin{enumerate}
\item $R$ and $M'$ are totally nonnegative, 
\item the matrix $R$ is representable by a cylindrical network $N_R$, and
\item the matrix $M'$ has strictly fewer nonzero entries than $M$, or is $(n', m')$-periodic for $n'<n$ or $m'<m$.
\end{enumerate}
By induction, $M'$ is also representable by a cylindrical network $N_{M'}$. It follows that $M$ is representable by a network obtained via concatenating $N_R$ and $N_{M'}$. The base case is an $(\ell,\ell)$-periodic matrix supported on a single diagonal. The latter is clearly representable by a cylindrical network $N$ with $\ell$ sources, $\ell$ sinks, and $\ell$ edges.

\textbf{Case I:} Suppose the matrix $M$ has rows or columns of all zeroes. Suppose, in particular, that the $k$th row of $M$ is all zeroes. Let $\mathrm{id}_k$ be the $(n-1,n)$-periodic matrix which is equal to the $(n,n)$-periodic identity matrix with all rows with indices equal to $k \bmod n$ removed. Then observe that $M = \mathrm{id}_k \cdot M'$, where $M'$ is equal to $M$, but with every row with index equal to $k \bmod n$ removed. In particular, $M'$ is $(n-1,m)$-periodic and TN. Since $M$ is periodic, we only need to apply this process finitely many times before we obtain a matrix with no rows of all zeroes. Now observe that $\mathrm{id}_k$ is equal to the weight matrix of the cylindrical network with $n-1$ sources $v_1, v_2, \ldots, v_{k-1}, v_{k+1}, \ldots, v_n$, with $n$ sinks $w_1, w_2, \ldots, w_n$, and with each source $v_i$ incident to a single edge of weight $1$ terminating at $w_i$ for all $i \neq k$.

If $M$ has columns of all zeroes, we apply a similar argument. When column $k$ of $M$ is all zeroes, we may write $M = M' \cdot (\mathrm{id}_k)^{\top}$, and $M'$ is equal to $M$ with the $k$th column removed. Similarly to before, $M'$ is $(n,n-1)$-periodic and TN. The matrix $(\mathrm{id}_k)^{\top}$ is representable by a cylindrical network. This is because, given a network $N$, one may reverse the orientations of all edges and exchange the roles of sources and sinks to obtain a network $N^\top$ satisfying $W(N^\top) = \left(W(N)\right)^\top$.

\textbf{Case II:} Suppose that $n=m$ and $M$ has no rows or columns of all zeroes. Then if $M$ does have a special SW corner, Proposition~\ref{lem:main} implies $M = R \cdot M'$, where $R$ eliminates the special SW corner of $M$ via periodic elementary row operations and $M'$ is still TN. The matrix $M'$ has strictly fewer nonzero entries, and $R$ is representable by Lemma~\ref{lem:elimination-representable}.

On the other hand, if $M$ does not have a special SW corner, it is upper triangular, and therefore invertible. Using this full-rank assumption, Theorem 2.6 of \cite{lam-pylyavskyy} allows us to eliminate a (non-special) SW corner of $M$ via periodic elementary row operations, and then follow reasoning similar to the special corner case.

\textbf{Case III:} Suppose $n \neq m$ and $M$ has no rows or columns of all zeroes. Then if $n > m$, we may apply Lemma~\ref{lem:SW-corner-existence} to find a special SW corner in $M$ and eliminate it with a periodic elementary row operation via Proposition~\ref{lem:main}, then proceed similarly to Case II.

If $n < m$, we apply Lemmas~\ref{lem:SW-corner-existence} and Proposition~\ref{lem:main} instead to $M^\top$, obtaining a factorization $M^\top = R \cdot M'$. Thus $M = (M')^\top \cdot R^\top$. Following the reasoning in Case I, the matrices $(M')^\top$ and $R^\top$ are representable by networks, since $M'$ and $R$ themselves are.
\end{proof}

\subsection{An Example of Proposition~\ref{lem:main}}\label{sec:eg}

Proving Proposition~\ref{lem:main} is the primary requirement for proving Theorem~\ref{thm:main}. The proof will require several new definitions. As such, we dedicate this section to computing an example first, to motivate these definitions. The full proof is contained in the following section.

Suppose $M$ is a $(2,1)$-periodic TN matrix:
    \[M = \begin{bmatrix}
    b_0 & b_1 & b_2 & b_3 & b_4 & \ldots & b_n & 0 & 0 & \ldots \\
    a_0 & a_1 & a_2 & a_3 & a_4 & \ldots & a_n & a_{n+1} & 0 & \ldots \\
    0 & b_0 & b_1 & b_2 & b_3 & \ldots & b_{n-1} & b_n & 0 & \ldots \\
    0 & a_0 & a_1 & a_2 & a_3 & \ldots & a_{n-1} & a_n & 0 & \ldots \\
    \vdots & \vdots & \vdots & \vdots & \vdots & \vdots & \vdots & \vdots & \vdots & \ddots
    \end{bmatrix}.\]
Then we obtain the following matrix $M'$ by eliminating a special SW corner of $M$:
    \[M' = \begin{bmatrix}
    b_0 & b_1 & b_2 & \ldots & b_{n-1} & b_n & 0 & \ldots \\
    a_0-cb_0 & a_1-cb_1 & a_2-cb_2 & \ldots & a_{n-1}-cb_{n-1} & a_n-cb_n & a_{n+1} & \ldots \\
    0 & b_0 & b_1 & \ldots & b_{n-2} & b_{n-1} & b_n & \ldots \\
    0 & a_0-cb_0 & a_1-cb_1 & \ldots & a_{n-2}-cb_{n-2} & a_{n-1}-cb_{n-1} & a_n-cb_n &\ldots \\
    \vdots & \vdots & \vdots & \vdots & \vdots & \vdots & \vdots & \ddots
    \end{bmatrix},\]
where $c = a_0/b_0$. Notice that we have chosen to display $a_0-cb_0$ in $M'$, even though it is equal to $0$. The utility of this will be made clear later.

When proving Proposition~\ref{lem:main}, we must show every submatrix of $M'$ has a nonnegative determinant. Consider the following submatrix of $M'$:
    \[\Delta = \begin{bmatrix}
    a_4-cb_4 & a_5-cb_5 & a_6-cb_6 \\
    b_2 & b_3 & b_4 \\
    a_1-cb_1 & a_2-cb_2 & a_3-cb_3
    \end{bmatrix}.\]
By the multilinearity of the determinant, we have
    \begin{align}
    \left|\Delta\right| &=\left| \begin{matrix}
    a_4-cb_4 & a_5-cb_5 & a_6-cb_6 \\
    b_2 & b_3 & b_4 \\
    a_1-cb_1 & a_2-cb_2 & a_3-cb_3
    \end{matrix} \right| \nonumber \\ 
    &=
    \left|
    \begin{matrix}
    a_4 & a_5 & a_6 \\
    b_2 & b_3 & b_4 \\
    a_1-cb_1 & a_2-cb_2 & a_3-cb_3
    \end{matrix} \right|
    - c \cdot \left|
    \begin{matrix}
    b_4 & b_5 & b_6 \\
    b_2 & b_3 & b_4 \\
    a_1-cb_1 & a_2-cb_2 & a_3-cb_3
    \end{matrix} \right| \nonumber \\ 
    &=
    \left|
    \begin{matrix}
    a_4 & a_5 & a_6 \\
    b_2 & b_3 & b_4 \\
    a_1 & a_2 & a_3
    \end{matrix} \right|
    - c \cdot \left|
    \begin{matrix}
    a_4 & a_5 & a_6 \\
    b_2 & b_3 & b_4 \\
    b_1 & b_2 & b_3
    \end{matrix} \right|
    - c \cdot \left|
    \begin{matrix}
    b_4 & b_5 & b_6 \\
    b_2 & b_3 & b_4 \\
    a_1 & a_2 & a_3
    \end{matrix} \right|
    + c^2 \cdot \left|
    \begin{matrix}
    b_4 & b_5 & b_6 \\
    b_2 & b_3 & b_4 \\
    b_1 & b_2 & b_3
    \end{matrix} \right|.
    \label{eq:example-multilin-expansion}\end{align}
Recalling that $c = a_0/b_0$, we may re-write the above as
    \begin{align}
    \left|
    \begin{matrix}
    a_4 & a_5 & a_6 \\
    b_2 & b_3 & b_4 \\
    a_1 & a_2 & a_3
    \end{matrix} \right|
    - \frac{a_0}{b_0} \cdot \left|
    \begin{matrix}
    a_4 & a_5 & a_6 \\
    b_2 & b_3 & b_4 \\
    b_1 & b_2 & b_3
    \end{matrix} \right|
    - \frac{a_0}{b_0} \cdot \left|
    \begin{matrix}
    b_4 & b_5 & b_6 \\
    b_2 & b_3 & b_4 \\
    a_1 & a_2 & a_3
    \end{matrix} \right|
    + \frac{a_0^2}{b_0^2} \cdot \left|
    \begin{matrix}
    b_4 & b_5 & b_6 \\
    b_2 & b_3 & b_4 \\
    b_1 & b_2 & b_3
    \end{matrix} \right|.
    \label{eq:example-with-denoms} \end{align}
Let us clear denominators in Expression~\ref{eq:example-with-denoms}, multiplying by $b_0^2 > 0$ to obtain
    \begin{align}
    b_0^2 \left|
    \begin{matrix}
    a_4 & a_5 & a_6 \\
    b_2 & b_3 & b_4 \\
    a_1 & a_2 & a_3
    \end{matrix} \right|
    - b_0a_0 \cdot \left|
    \begin{matrix}
    a_4 & a_5 & a_6 \\
    b_2 & b_3 & b_4 \\
    b_1 & b_2 & b_3
    \end{matrix} \right|
    - b_0a_0 \cdot \left|
    \begin{matrix}
    b_4 & b_5 & b_6 \\
    b_2 & b_3 & b_4 \\
    a_1 & a_2 & a_3
    \end{matrix} \right|
    + a_0^2 \cdot \left|
    \begin{matrix}
    b_4 & b_5 & b_6 \\
    b_2 & b_3 & b_4 \\
    b_1 & b_2 & b_3
    \end{matrix} \right|.
    \label{eq:example-cleared}
    \end{align}

Now, consider the following submatrix of $M$:
    \[\tilde{\Delta} = \begin{bmatrix}
    b_0 & b_3 & b_4 & b_5 & b_6 \\
    a_0 & a_3 & a_4 & a_5 & a_6 \\
    0 & b_1 & b_2 & b_3 & b_4 \\
    0 & b_0 & b_1 & b_2 & b_3 \\
    0 & a_0 & a_1 & a_2 & a_3
    \end{bmatrix}.\]
Since $M$ is totally nonnegative, so is $\tilde{\Delta}$. Observe that Expression~\ref{eq:example-cleared} is equal to
    \begin{align}
    \left|\tilde{\Delta}[14,12]\right| \cdot \left|\tilde{\Delta}[235,345]\right|
    -&\left|\tilde{\Delta}[15,12]\right| \cdot \left|\tilde{\Delta}[234,345]\right| \label{eq:signed-comps} \\
    &-\left|\tilde{\Delta}[24,12]\right| \cdot \left|\tilde{\Delta}[135,345]\right|
    +\left|\tilde{\Delta}[25,12]\right| \cdot \left|\tilde{\Delta}[134,345]\right|.
    \nonumber
    \end{align}
In other words, each of the summands in Expression~\ref{eq:example-cleared} is (up to a sign) a complementary minor immanant of $\tilde{\Delta}$. Via Definition~\ref{dfn:comp-minor-diagram}, we re-write Expression~\ref{eq:signed-comps} as
    \begin{align}
    \begin{tikzpicture}[baseline,x=.75cm, y=.5cm]
\begin{scope}[shift={(0,2.1)}]
\draw[circle, fill=black](0, 0) circle[radius = 1mm] node {};
\draw[circle, fill=black](0, -1) circle[radius = 1mm] node {};
\draw[circle, fill=white](0, -2) circle[radius = 1mm] node {};
\draw[circle, fill=white](0, -3) circle[radius = 1mm] node {};
\draw[circle, fill=white](0, -4) circle[radius = 1mm] node {};
\draw[circle, fill=white](1, 0) circle[radius = 1mm] node {};
\draw[circle, fill=black](1, -1) circle[radius = 1mm] node {};
\draw[circle, fill=black](1, -2) circle[radius = 1mm] node {};
\draw[circle, fill=white](1, -3) circle[radius = 1mm] node {};
\draw[circle, fill=black](1, -4) circle[radius = 1mm] node {};
\end{scope}
\end{tikzpicture}
\quad - \quad
\begin{tikzpicture}[baseline,x=.75cm, y=.5cm]
\begin{scope}[shift={(0,2.1)}]
\draw[circle, fill=black](0, 0) circle[radius = 1mm] node {};
\draw[circle, fill=black](0, -1) circle[radius = 1mm] node {};
\draw[circle, fill=white](0, -2) circle[radius = 1mm] node {};
\draw[circle, fill=white](0, -3) circle[radius = 1mm] node {};
\draw[circle, fill=white](0, -4) circle[radius = 1mm] node {};
\draw[circle, fill=white](1, 0) circle[radius = 1mm] node {};
\draw[circle, fill=black](1, -1) circle[radius = 1mm] node {};
\draw[circle, fill=black](1, -2) circle[radius = 1mm] node {};
\draw[circle, fill=black](1, -3) circle[radius = 1mm] node {};
\draw[circle, fill=white](1, -4) circle[radius = 1mm] node {};
\end{scope}
\end{tikzpicture}
\quad - \quad
\begin{tikzpicture}[baseline,x=.75cm, y=.5cm]
\begin{scope}[shift={(0,2.1)}]
\draw[circle, fill=black](0, 0) circle[radius = 1mm] node {};
\draw[circle, fill=black](0, -1) circle[radius = 1mm] node {};
\draw[circle, fill=white](0, -2) circle[radius = 1mm] node {};
\draw[circle, fill=white](0, -3) circle[radius = 1mm] node {};
\draw[circle, fill=white](0, -4) circle[radius = 1mm] node {};
\draw[circle, fill=black](1, 0) circle[radius = 1mm] node {};
\draw[circle, fill=white](1, -1) circle[radius = 1mm] node {};
\draw[circle, fill=black](1, -2) circle[radius = 1mm] node {};
\draw[circle, fill=white](1, -3) circle[radius = 1mm] node {};
\draw[circle, fill=black](1, -4) circle[radius = 1mm] node {};
\end{scope}
\end{tikzpicture}
\quad + \quad
\begin{tikzpicture}[baseline,x=.75cm, y=.5cm]
\begin{scope}[shift={(0,2.1)}]
\draw[circle, fill=black](0, 0) circle[radius = 1mm] node {};
\draw[circle, fill=black](0, -1) circle[radius = 1mm] node {};
\draw[circle, fill=white](0, -2) circle[radius = 1mm] node {};
\draw[circle, fill=white](0, -3) circle[radius = 1mm] node {};
\draw[circle, fill=white](0, -4) circle[radius = 1mm] node {};
\draw[circle, fill=black](1, 0) circle[radius = 1mm] node {};
\draw[circle, fill=white](1, -1) circle[radius = 1mm] node {};
\draw[circle, fill=black](1, -2) circle[radius = 1mm] node {};
\draw[circle, fill=black](1, -3) circle[radius = 1mm] node {};
\draw[circle, fill=white](1, -4) circle[radius = 1mm] node {};
\end{scope}
\end{tikzpicture}. \label{eq:main-eg-tls}
    \end{align}
Now apply Lemma~\ref{lem:tl-compat} to each complementary minor diagram above to obtain a signed sum of Temperley-Lieb immanants of $\tilde{\Delta}$:
    \input{Tikz-Pictures/expanded-tl-example}
Keeping in mind that the coloring does not affect the value of a Temperley-Lieb immanant, many of the terms in the above sum will cancel. After cancellation, we have
    \begin{align}\begin{tikzpicture}[baseline,x=.75cm, y=.5cm]
    \begin{scope}[shift={(0,3)}]
    \draw (0, -1) to (1, -1);
    \draw[circle, fill=black](0, -1) circle[radius = 1mm] node {};
    \draw[circle, fill=white](1, -1) circle[radius = 1mm] node {};
    \draw (0, -2) to [bend left=45](0, -3);
    \draw[circle, fill=black](0, -2) circle[radius = 1mm] node {};
    \draw[circle, fill=white](0, -3) circle[radius = 1mm] node {};
    \draw (1, -2) to (0, -4);
    \draw[circle, fill=white](0, -4) circle[radius = 1mm] node {};
    \draw[circle, fill=black](1, -2) circle[radius = 1mm] node {};
    \draw (0, -5) to (1, -5);
    \draw[circle, fill=white](0, -5) circle[radius = 1mm] node {};
    \draw[circle, fill=black](1, -5) circle[radius = 1mm] node {};
    \draw (1, -3) to [bend right=45](1, -4);
    \draw[circle, fill=white](1, -4) circle[radius = 1mm] node {};
    \draw[circle, fill=black](1, -3) circle[radius = 1mm] node {};
    \end{scope}
    \end{tikzpicture}
    \quad - \quad
    \begin{tikzpicture}[baseline,x=.75cm, y=.5cm]
    \begin{scope}[shift={(0,3)}]
    \draw (0, -1) to [bend left=45](0, -4);
    \draw[circle, fill=black](0, -1) circle[radius = 1mm] node {};
    \draw[circle, fill=white](0, -4) circle[radius = 1mm] node {};
    \draw (0, -2) to [bend left=45](0, -3);
    \draw[circle, fill=black](0, -2) circle[radius = 1mm] node {};
    \draw[circle, fill=white](0, -3) circle[radius = 1mm] node {};
    \draw (0, -5) to (1, -5);
    \draw[circle, fill=white](0, -5) circle[radius = 1mm] node {};
    \draw[circle, fill=black](1, -5) circle[radius = 1mm] node {};
    \draw (1, -1) to [bend right=45](1, -4);
    \draw[circle, fill=white](1, -4) circle[radius = 1mm] node {};
    \draw[circle, fill=black](1, -1) circle[radius = 1mm] node {};
    \draw (1, -2) to [bend right=45](1, -3);
    \draw[circle, fill=black](1, -3) circle[radius = 1mm] node {};
    \draw[circle, fill=white](1, -2) circle[radius = 1mm] node {};
    \end{scope}
    \end{tikzpicture}
    \quad - \quad
    \begin{tikzpicture}[baseline,x=.75cm, y=.5cm]
    \begin{scope}[shift={(0,3)}]
    \draw (0, -1) to [bend left=45](0, -4);
    \draw[circle, fill=black](0, -1) circle[radius = 1mm] node {};
    \draw[circle, fill=white](0, -4) circle[radius = 1mm] node {};
    \draw (0, -2) to [bend left=45](0, -3);
    \draw[circle, fill=black](0, -2) circle[radius = 1mm] node {};
    \draw[circle, fill=white](0, -3) circle[radius = 1mm] node {};
    \draw (1, -1) to (0, -5);
    \draw[circle, fill=white](0, -5) circle[radius = 1mm] node {};
    \draw[circle, fill=black](1, -1) circle[radius = 1mm] node {};
    \draw (1, -2) to [bend right=45](1, -5);
    \draw[circle, fill=black](1, -5) circle[radius = 1mm] node {};
    \draw[circle, fill=white](1, -2) circle[radius = 1mm] node {};
    \draw (1, -3) to [bend right=45](1, -4);
    \draw[circle, fill=white](1, -4) circle[radius = 1mm] node {};
    \draw[circle, fill=black](1, -3) circle[radius = 1mm] node {};
    \end{scope}
    \end{tikzpicture}.\label{eq:remaining-terms}\end{align}
Our goal is to show that this sum is nonnegative. By Theorem~\ref{lem:tl-pos}, each Temperley-Lieb immanant is nonnegative individually, but two of them appear here with a minus sign. Notice, however, that the middle term is also  $(34,12)$-compatible, and
    \[\tilde{\Delta}[34, 12] \cdot \tilde{\Delta}[125,345] = 
    \left|\begin{matrix}
    0 & b_1 \\
    0 & b_0
    \end{matrix}\right| \cdot 
    \left|\begin{matrix}
    b_4 & b_5 & b_6 \\
    a_4 & a_5 & a_6 \\
    a_1 & a_2 & a_3
    \end{matrix}\right| = 0.\]
Thus, by Corollary~\ref{lem:tl-zero}, the middle Temperley-Lieb immanant in Expression~\ref{eq:remaining-terms} is $0$. A similar argument shows that the third term in Expression~\ref{eq:remaining-terms} is also $0$. The only remaining Temperley-Lieb immanant has a positive coefficient, so $b_0^2|\Delta| \geq 0$, and therefore $\left|\Delta\right| \geq 0$.

\subsection{The Proof of Proposition~\ref{lem:main}}\label{sec:proof-of-main-lemma}

Our goal in this section is to describe general versions of the calculations found in the previous section. For the remainder of this section, we let $M$ be some fixed TN $(n,m)$-periodic infinite matrix and let $M'$ be obtained from $M$ by eliminating a special SW corner at entry $(i_\star, j_\star)$. We also fix $a_0 = M_{i^\star, j^\star}$ and $b_0 =  M_{i^\star-1, j^\star}$, so that $c = a_0 / b_0$ is the constant used in the periodic elementary row operation achieving this elimination. Label the nonzero entries in row $i^\star$ as $a_0, a_1, a_2, \ldots$ and label the nonzero entries in row $i^\star-1$ as $b_0, b_1, b_2, \ldots$. 

Given a submatrix $\Delta$ of $M'$, we build a generalized submatrix $\tilde{\Delta}$ of $M$. We then show that $|\Delta|$ is a nonnegative sum of Temperley-Lieb immanants of $\tilde{\Delta}$. Since $M$ is assumed to be TN, we know that $\tilde{\Delta}$ is TN as well, so its Temperley-Lieb immanants are nonnegative. The matrix $\tilde{\Delta}$ is built specifically to reflect the expansion of $|\Delta|$ via multilinearity, as in Expressions~\ref{eq:example-multilin-expansion}~and~\ref{eq:signed-comps}.

\begin{dfn}
For a set $X \subseteq \mathbb{N}$, let $X^- = \{x - 1 : x \in X\}$.
\end{dfn}

\begin{dfn}\label{dfn:delta-tilde}
Given a submatrix $\Delta = M'\left[I,J\right]$, let
\[I^\star = I \cap (\mathbb{N}i^\star), \quad A(\Delta) = \{r \in I^\star :r-1 \not\in I\}, \quad C(\Delta) = \left\{j^\star + km:i^\star + kn \in A(\Delta)\right\}.\] Then define 
\[\tilde{\Delta} = M[I \cup A(\Delta)^-, J \sqcup C(\Delta)],\] 
where $\sqcup$ is a disjoint union. Thus, in general, $J \sqcup C(\Delta)$ may contain duplicate elements, and as such, $\tilde{\Delta}$ is a generalized subamtrix of $M$.
\end{dfn}

Intuitively, $A(\Delta)$ is the set of rows of $\Delta$ which need to be `dealt with' when proving that $\Delta$ is TN. If $r \not \in I^\star$, then $r$ is already equal to a row of $M$. Additionally, if $r \in I^\star$ and $r-1 \in I$, then we can undo the row operation within $\Delta$ while preserving $|\Delta|$.
The purpose of the columns $C(\Delta)$ is to provide additional copies of $a_0$ and $b_0$ in the matrix $\tilde{\Delta}$. These will be used to represent clearing denominators after expanding $|\Delta|$ via multilinearity; see Expressions~\ref{eq:example-multilin-expansion}~and~\ref{eq:example-cleared}.

\begin{ex}\label{eg:largereg-tilde}
Suppose $M$ is a $(4,2)$-periodic TN matrix with a special SW corner at $(4,1)$. Then we have
\[M = \begin{bmatrix}
f_0 & f_1 & f_2 & f_3 & \ldots \\
d_0 & d_1 & d_2 & d_3 & \ldots \\
b_0 & b_1 & b_2 & b_3 & \ldots \\
a_0 & a_1 & a_2 & a_3 & \ldots \\
0 & 0 & f_0 & f_1 & \ldots \\
0 & 0 & d_0 & d_1 & \ldots \\
0 & 0 & b_0 & b_1 & \ldots \\
0 & 0 & a_0 & b_1 & \ldots \\
\vdots & \vdots & \vdots & \vdots & \ddots \\
\end{bmatrix},
\quad
M' = \begin{bmatrix}
f_0 & f_1 & f_2 & f_3 & \ldots \\
d_0 & d_1 & d_2 & d_3 & \ldots \\
b_0 & b_1 & b_2 & b_3 & \ldots \\
a_0-cb_0 & a_1-cb_1 & a_2-cb_2 & a_3-cb_3 & \ldots \\
0 & 0 & f_0 & f_1 & \ldots \\
0 & 0 & d_0 & d_1 & \ldots \\
0 & 0 & b_0 & b_1 & \ldots \\
0 & 0 & a_0-cb_0 & a_1-cb_1 & \ldots \\
\vdots & \vdots & \vdots & \vdots & \ddots \\
\end{bmatrix}.\]

\newpage

Suppose

\[
\Delta = M'[\{4,6,8,12\}, \{2,3,4,5\}] = 
\begin{blockarray}{ccccc}
    & 2 & 3 & 4 & 5 \\
    \begin{block}{c[cccc]}
4 & a_1-cb_1 & a_2-cb_2 & a_3-cb_3 & a_4-cb_4 \\
6 & 0 & d_0 & d_1 & d_2 \\
8 & 0 & a_0-cb_0 & a_1-cb_1 & a_2-cb_2 \\
12 & 0 & 0 & 0 & a_0-cb_0 \\
    \end{block}
\end{blockarray}.
\]

Then $A(\Delta) = \{4,8,12\}$ and $C(\Delta) = \{1,3,5\}$, so

\begin{align*}
\tilde{\Delta} = M[\{3,4,6,7,8,11,12\}, \{1,2,3,3,4,5,5\}] 
= \begin{blockarray}{cccccccc}
& 1 & 2 & 3 & 3 & 4 & 5 & 5 \\
    \begin{block}{c[ccccccc]}
3 & b_0 & b_1 & b_2 & b_2 & b_3 & b_4 & b_4 \\
4 & a_0 & a_1 & a_2 & a_2 & a_3 & a_4 & a_4 \\
6 & 0 & 0 & d_0 & d_0 & d_1 & d_2 & d_2 \\
7 & 0 & 0 & b_0 & b_0 & b_1 & b_2 & b_2 \\
8 & 0 & 0 & a_0 & a_0 & a_1 & a_2 & a_2 \\
11 & 0 & 0 & 0 & 0 & 0 & b_0 & b_0 \\
12 & 0 & 0 & 0 & 0 & 0 & a_0 & a_0 \\
    \end{block}
\end{blockarray}.
\end{align*}

\end{ex}

\begin{dfn}\label{dfn:delta-s}
Given a submatrix $\Delta = M'\left[I,J\right]$ and $S \subseteq A(\Delta)$, define \[\Delta^S = M[(I \setminus S) \cup S^-, \,\, J].\]
\end{dfn}

Intuitively, the matrix $\Delta^S$ is a `specialized' version of $\Delta$. Each row $r \in A(\Delta)$ has been changed out for the corresponding row of either all $a_i$'s or all $b_i$'s. If $r \in S$, we place the row of all $b_i$s; otherwise, we place the row of all $a_i$s.

\begin{ex}\label{eg:largereg-delta-S}
Continuing with the matrices from the Example~\ref{eg:largereg-tilde}, we have
\[
\Delta^{\emptyset} = \begin{bmatrix}
a_1 & a_2 & a_3 & a_4 \\
0 & d_0 & d_1 & d_2 \\
0 & a_0 & a_1 & a_2 \\
0 & 0 & 0 & a_0 \\
\end{bmatrix},\,\,
\Delta^{\{8\}} = \begin{bmatrix}
a_1 & a_2 & a_3 & a_4 \\
0 & d_0 & d_1 & d_2 \\
0 & b_0 & b_1 & b_2 \\
0 & 0 & 0 & a_0 \\
\end{bmatrix},\,\,
\Delta^{\{4,12\}} = \begin{bmatrix}
b_1 & b_2 & b_3 & b_4 \\
0 & d_0 & d_1 & d_2 \\
0 & a_0 & a_1 & a_2 \\
0 & 0 & 0 & b_0 \\
\end{bmatrix}.
\]
\end{ex}

\begin{lem}\label{lem:multilin-expansion}
For any submatrix $\Delta$ of $M'$, we have $\left|\Delta\right| = \sum_{S \subseteq A(\Delta)} (-c)^{|S|} \left|\Delta^S\right|.$
\end{lem}

\begin{proof}
When expanding $|\Delta|$ via multilinearity along the rows in $A(\Delta)$ as in Equation~\ref{eq:example-multilin-expansion}, we must choose either the $a_i$ or $b_i$ specialization of that row. Each time we select $b_i$ this way, we introduce a factor of $-c$. Thus $(-c)^{|S|} \left|\Delta^S\right|$ is a term that appears in the multilinear expansion of $|\Delta|$, and summing over all possible choices of $S \subset A(\Delta)$, we obtain the statement of the lemma.
\end{proof}

\begin{lem}\label{lem:multilin-expansion-is-tilde-cm}
For any submatrix $\Delta$ of $M$ and any $S \subset A(\Delta)$, we have
\begin{align}
b_0^{|A(\Delta)|} (-c)^{|S|} = a_0^{|S|} b_0^{|A(\Delta)|-|S|} = \left|\tilde{\Delta}[S \cup (A(\Delta) \setminus S)^-, \,\, C(\Delta)]\right|. \label{eq:diagonal}
\end{align}
In particular, $b_0^{|A(\Delta)|} (-c)^{|S|} \left|\Delta^S\right|$ is a complimentary minor immanant of $\tilde{\Delta}$.
\end{lem}

\begin{proof}
First, we defend the claim about complimentary minor immanants. Recall from Definition~\ref{dfn:delta-s} that $\Delta^S = M[(I \setminus S) \cup S^-, \,\, J]$. Then note that
\[(I \setminus S) \cup S^- \quad \text{and} \quad S \cup (A(\Delta) \setminus S)^-\] 
partition the set $I \cup A(\Delta)^-$, which is precisely the row set of $\tilde{\Delta}$ by Definition~\ref{dfn:delta-tilde}. Similarly, $C(\Delta)$ and $J$ partition the column set of $\tilde{\Delta}$. Thus, once we prove Equation~\ref{eq:diagonal} above, Lemma~\ref{lem:multilin-expansion} will imply that $|\Delta|$ is equal to $b_0^{|A(\Delta)|}$ times a signed sum of complimentary minor immanants of $\tilde{\Delta}$.

For Equation~\ref{eq:diagonal}, we claim that $\tilde{\Delta}[S \cup (A(\Delta) \setminus S)^-, \,\, C(\Delta)]$ is an upper triangular matrix whose main diagonal contains $|S|$ copies of $a_0$ and $|A(\Delta)|-|S|$ copies of $b_0$. The claim about the value of its determinant will follow immediately. By Definition~\ref{dfn:delta-tilde}, the set $C(\Delta)$ contains only indices which are equal to $j^\star$ modulo $m$. Similarly, the elements of $S \cup (A(\Delta) \setminus S)^-$ are all either equal to $i^\star$ or $i^\star-1$ modulo $n$. Since $M_{i^\star, j^\star} = a_0$ and $M_{i^\star-1, j^\star} = b_0$, it follows that the diagonal entries of $\tilde{\Delta}[S \cup (A(\Delta) \setminus S)^-, \,\, C(\Delta)]$ are all either $a_0$ or $b_0$. In addition, since $(i^\star, j^\star)$ is a special SW corner of $M$, all entries below and to the left of the diagonal must be $0$. The only possible exception would be an $a_0$ immediately below a $b_0$. By construction, however, if $r-1 \in (A(\Delta) \setminus S)^-$, then $r$ is in $A(\Delta) \setminus S \subseteq (I \setminus S)$, the complementary row set. This completes the proof.
\end{proof}

\begin{ex}
We continue from Example~\ref{eg:largereg-tilde}. If $S=\emptyset$, then 
\[S \cup (A(\Delta) \setminus S)^- = A(\Delta)^- = \{3,7,11\} \quad \text{and} \quad C(\Delta) = \{1,3,5\}.\]
We have
\begin{align*}
\tilde{\Delta}[\{3,7,11\},\{1,3,5\}] = 
\begin{blockarray}{cccccccc}
& \circled{1} & 2 & \circled{3} & 3 & 4 & \circled{5} & 5 \\
\begin{block}{c[ccccccc]}
\circled{3} & b_0 & b_1 & b_2 & b_2 & b_3 & b_4 & b_4 \\
4 & a_0 & a_1 & a_2 & a_2 & a_3 & a_4 & a_4 \\
6 & 0 & 0 & d_0 & d_0 & d_1 & d_2 & d_2 \\
\circled{7} & 0 & 0 & b_0 & b_0 & b_1 & b_2 & b_2 \\
8 & 0 & 0 & a_0 & a_0 & a_1 & a_2 & a_2 \\
\circled{11} & 0 & 0 & 0 & 0 & 0 & b_0 & b_0 \\
12 & 0 & 0 & 0 & 0 & 0 & a_0 & a_0 \\
\end{block}
\end{blockarray} = \begin{bmatrix}
    b_0 & b_2 & b_4 \\
    0 & b_0 & b_2 \\
    0 & 0 & b_0 \\
    \end{bmatrix}.
\end{align*}
Clearly, the determinant of the above is $a_0^{|S|} b_0^{|A(\Delta)|-|S|}=b_0^3$. The complementary minor is 
\begin{align*}
\tilde{\Delta}[\{4,6,8,12\},\{2,3,4,5\}] &=
\begin{blockarray}{cccccccc}
& 1 & \circled{2} & 3 & \circled{3} & \circled{4} & 5 & \circled{5} \\
\begin{block}{c[ccccccc]}
3 & b_0 & b_1 & b_2 & b_2 & b_3 & b_4 & b_4 \\
\circled{4} & a_0 & a_1 & a_2 & a_2 & a_3 & a_4 & a_4 \\
\circled{6} & 0 & 0 & d_0 & d_0 & d_1 & d_2 & d_2 \\
7 & 0 & 0 & b_0 & b_0 & b_1 & b_2 & b_2 \\
\circled{8} & 0 & 0 & a_0 & a_0 & a_1 & a_2 & a_2 \\
11 & 0 & 0 & 0 & 0 & 0 & b_0 & b_0 \\
\circled{12} & 0 & 0 & 0 & 0 & 0 & a_0 & a_0 \\
\end{block}
\end{blockarray}
= \begin{bmatrix}
    a_1 & a_2 & a_3 & a_4 \\
    0 & d_0 & d_1 & d_2 \\
    0 & a_0 & a_1 & a_2 \\
    0 & 0 & 0 & a_0 \\
    \end{bmatrix} = \Delta^\emptyset.
\end{align*}
\end{ex}

\begin{dfn}\label{dfn:dcmd}
A \textit{decorated complementary minor diagram} is a complementary minor diagram with \textit{gray blocks} inserted, which satisfy the relation
    \[\begin{tikzpicture}[baseline,x=.75cm, y=.5cm]
    \begin{scope}[shift={(0,.55)}]
    \draw[circle, fill=gray] (0, 0) circle[radius = 1mm] node (D1) {};
    \draw[circle, fill=gray] (0, -1) circle[radius = 1mm] node (D2) {};
    \node[draw, dotted, fit=(D1) (D2), inner sep=0mm] {};
    \end{scope}
    \end{tikzpicture}
    =
    \begin{tikzpicture}[baseline,x=.75cm, y=.5cm]
    \begin{scope}[shift={(0,.55)}]
    \draw[circle, fill=white] (0, 0) circle[radius = 1mm] node (D1) {};
    \draw[circle, fill=black] (0, -1) circle[radius = 1mm] node (D2) {};
    \node[draw, dotted, fit=(D1) (D2), inner sep=0mm] {};
    \end{scope}
    \end{tikzpicture}
    -
    \begin{tikzpicture}[baseline,x=.75cm, y=.5cm]
    \begin{scope}[shift={(0,.55)}]
    \draw[circle, fill=black] (0, 0) circle[radius = 1mm] node (D1) {};
    \draw[circle, fill=white] (0, -1) circle[radius = 1mm] node (D2) {};
    \node[draw, dotted, fit=(D1) (D2), inner sep=0mm] {};
    \end{scope}
    \end{tikzpicture}\]
inside a larger complementary minor diagram. Given a decorated complementary minor diagram $\mathcal{D}$, we define $\mathrm{GB}(\mathcal{D})$ to be its set of gray blocks.
\end{dfn}

\begin{ex}\label{rmk:gb} Let us apply the relation in Definition~\ref{dfn:dcmd} to a decorated complementary minor diagram with two gray blocks:
    \input{Tikz-Pictures/gb-eg}
Each decorated complementary minor diagram is a signed sum of ordinary complementary minor diagrams. As such, if $\mathcal{D}$ is a decorated complementary minor diagram, we will write $\mathcal{D}(M)$ to mean the evaluation of the corresponding linear combination complementary minor immanants on the matrix $M$.
\end{ex}

\begin{dfn}\label{dfn:sdcmd}
If $\mathcal{D}$ is a decorated complementary minor diagram, for any $S \subseteq \mathrm{GB}(\mathcal{D})$, define $\mathcal{D}^S$ to be the (undecorated) complementary minor diagram wherein each gray block in $S$ is replaced with
    \begin{tikzpicture}[baseline,x=.75cm, y=.5cm]
    \begin{scope}[shift={(0,.55)}]
    \draw[circle, fill=black] (0, 0) circle[radius = 1mm] node (D1) {};
    \draw[circle, fill=white] (0, -1) circle[radius = 1mm] node (D2) {};
    \node[draw, dotted, fit=(D1) (D2), inner sep=0mm] {};
    \end{scope}
    \end{tikzpicture}
and each gray block not in $S$ is replaced with
    \begin{tikzpicture}[baseline,x=.75cm, y=.5cm]
    \begin{scope}[shift={(0,.55)}]
    \draw[circle, fill=white] (0, 0) circle[radius = 1mm] node (D1) {};
    \draw[circle, fill=black] (0, -1) circle[radius = 1mm] node (D2) {};
    \node[draw, dotted, fit=(D1) (D2), inner sep=0mm] {};
    \end{scope}
    \end{tikzpicture}.
\end{dfn}
    
\begin{cor}\label{cor:dcmd-sum}
By repeatedly applying the relation in Definition~\ref{dfn:dcmd}, we have
    \[\mathcal{D} = \sum_{S \subseteq \mathrm{GB}(\mathcal{D})} (-1)^{|S|} \mathcal{D}^S.\]
\end{cor}

\begin{dfn}\label{dfn:d-delta}
Given a submatrix $\Delta=M'[I,J]$ of $M'$, define a decorated complementary minor diagram $\mathcal{D}_{\Delta}$ as follows: first, order the elements of $J \sqcup C(\Delta)$ so that $J \sqcup C = \{c_1 \leq c_2 \leq \ldots \leq c_\ell\}$. Then, to create the left column of $\mathcal{D}_{\Delta}$, do the following for each $1 \leq i \leq \ell$, starting from the top and moving down the column:
\begin{itemize}
\item if $c_i \in C(\Delta) \setminus J$, place a black dot;
\item if $c_i \in J \setminus C(\Delta)$, place a white dot;
\item if $c_i \in J \cap C(\Delta)$ and $c_i \neq c_{i-1}$, place a black dot; and
\item if $c_i \in J \cap C(\Delta)$ and $c_i = c_{i-1}$, place a white dot.
\end{itemize}
Also order the elements of $I \cup A(\Delta)^-$ so that $I \cup A(\Delta)^- = \{r_1 \leq r_2 \leq \ldots \leq r_\ell\}$. Then, to create the right column of $\mathcal{D}_{\Delta}$, do the following for each $1 \leq i \leq \ell$, starting from the top and moving down the column:
\begin{itemize}
\item if $r_i \not\in A(\Delta) \cup A(\Delta)^-$, place a black dot; and
\item if $r_{i} \in A(\Delta)^-$ (and therefore $r_{i+1} \in A(\Delta)$), place a gray block.
\end{itemize}
\end{dfn}

\begin{ex}
\[
\tilde{\Delta} = 
\begin{blockarray}{ccccccccc}
& & \blackdot & \whitedot & \blackdot & \whitedot & \whitedot & \blackdot & \whitedot \\
& & 1 & 2 & 3 & 3 & 4 & 5 & 5 \\
    \begin{block}{cc[ccccccc]}
\topgraydot & 3 & b_0 & b_1 & b_2 & b_2 & b_3 & b_4 & b_4 \\
\bottomgraydot & 4 & a_0 & a_1 & a_2 & a_2 & a_3 & a_4 & a_4 \\
\blackdot & 6 &  0 & 0 & d_0 & d_0 & d_1 & d_2 & d_2 \\
\topgraydot & 7 & 0 & 0 & b_0 & b_0 & b_1 & b_2 & b_2 \\
\bottomgraydot & 8 & 0 & 0 & a_0 & a_0 & a_1 & a_2 & a_2 \\
\topgraydot & 11 & 0 & 0 & 0 & 0 & 0 & b_0 & b_0 \\
\bottomgraydot & 12 & 0 & 0 & 0 & 0 & 0 & a_0 & a_0 \\
    \end{block}
\end{blockarray}
\quad \rightsquigarrow \quad 
\mathcal{D}_{\Delta} = 
\begin{tikzpicture}[baseline,x=.75cm, y=.5cm]
\begin{scope}[shift={(0,4)}]
\draw[circle, fill=black](0, -1) circle[radius = 1mm] node {};
\draw[circle, fill=white](0, -2) circle[radius = 1mm] node {};
\draw[circle, fill=black](0, -3) circle[radius = 1mm] node {};
\draw[circle, fill=white](0, -4) circle[radius = 1mm] node {};
\draw[circle, fill=white](0, -5) circle[radius = 1mm] node {};
\draw[circle, fill=black](0, -6) circle[radius = 1mm] node {};
\draw[circle, fill=white](0, -7) circle[radius = 1mm] node {};

\draw[circle, fill=gray] (1, -1) circle[radius = 1mm] node (D1) {};
\draw[circle, fill=gray] (1, -2) circle[radius = 1mm] node (D2) {};
\draw[circle, fill=black](1, -3) circle[radius = 1mm] node {};
\draw[circle, fill=gray] (1, -4) circle[radius = 1mm] node (D3) {};
\draw[circle, fill=gray] (1, -5) circle[radius = 1mm] node (D4) {};
\draw[circle, fill=gray] (1, -6) circle[radius = 1mm] node (D5) {};
\draw[circle, fill=gray] (1, -7) circle[radius = 1mm] node (D6) {};
\node[draw, dotted, fit=(D1) (D2), inner sep=0mm] {};
\node[draw, dotted, fit=(D3) (D4), inner sep=0mm] {};
\node[draw, dotted, fit=(D5) (D6), inner sep=0mm] {};
\end{scope}
\end{tikzpicture}
\]
\end{ex}

\begin{lem}\label{lem:dcmd-eq-det}
For any submatrix $\Delta$ of $M'$, we have $\left| \Delta\right| = {b_0}^{-|A(\Delta)|} \mathcal{D}_{\Delta}\left(\tilde{\Delta}\right)$.
\end{lem}

\begin{proof}
By Definition~\ref{dfn:delta-s} and Lemmas~\ref{lem:multilin-expansion}~and~\ref{lem:multilin-expansion-is-tilde-cm}, we have
\begin{align*}
\left|\Delta\right| &= \sum_{S \subseteq A(\Delta)} (-c)^{|S|} \left|\Delta^S\right| \\
&= \sum_{S \subseteq A(\Delta)} (-1)^{|S|} a^{|S|} b^{-|S|} \left|\tilde{\Delta}[(I \setminus S) \cup S^-, \,\, J]\right| \\
&= b^{-|A(\Delta)|} \sum_{S \subseteq A(\Delta)} (-1)^{|S|} a^{|S|} b^{|A|-|S|} \left|\tilde{\Delta}[(I \setminus S) \cup S^-, \,\, J]\right| \\
&= b^{-|A(\Delta)|} \sum_{S \subseteq A(\Delta)} (-1)^{|S|}  \left|\tilde{\Delta}[S \cup (A(\Delta) \setminus S)^-, \,\, C(\Delta)]\right| \left|\tilde{\Delta}[(I \setminus S) \cup S^-, \,\, J]\right| \\
&= b^{-|A(\Delta)|} \sum_{S \subseteq A(\Delta)} (-1)^{|S|}  \mathcal{C}_{S \cup (A(\Delta) \setminus S)^-, \,\, C(\Delta)}\left(\tilde{\Delta}\right).
\end{align*}

Now, by Definition~\ref{dfn:d-delta}, we know that $A(\Delta)$ corresponds exactly to $\mathrm{GB}(\mathcal{D}_{\Delta})$. Thus each $S \subseteq A(\Delta)$ can be identified with a subset $S \subseteq \mathrm{GB}(\mathcal{D}_{\Delta})$. Given such a subset $S$, we claim that
    \begin{align}
    \mathcal{C}_{S \cup (A(\Delta) \setminus S)^-, \,\,\, C(\Delta)}\left(\tilde{\Delta}\right) = \mathcal{D}_{\Delta}^S\left(\tilde{\Delta}\right). \label{eq:comp-is-dcmd}
    \end{align}
To see that this is true, we check that the complementary minor diagrams for each side are the same.

For each side of Equation~\ref{eq:comp-is-dcmd}, the left column of the associated complementary minor diagram has a black dot for each member of $C(\Delta)$ and a white dot for each member of $J$; this is true regardless of the choice of $S$, since the left column does not depend on $S$ in either immanant.

Now, let us consider the right column of the complementary minor diagrams for the immanants in Equation~\ref{eq:comp-is-dcmd}. By definition, $\mathcal{D}_{\Delta}$ has a gray block for each $r \in A(\Delta)$; this gray block is associated to a pair $\{r-1, r\}$ of rows in $\tilde{\Delta}$. When we pass to $\mathcal{D}_{\Delta}^S$, if $r \in S$, the corresponding gray block is specialized to       
    \[\begin{tikzpicture}[baseline,x=.75cm, y=.5cm]
    \begin{scope}[shift={(0,.6)}]
    \draw[circle, fill=black] (0, 0) circle[radius = 1mm] node (D1) {};
    \draw[circle, fill=white] (0, -1) circle[radius = 1mm] node (D2) {};
    \node[draw, dotted, fit=(D1) (D2), inner sep=0mm] {};
    \end{scope}
    \end{tikzpicture}.\]
This corresponds to adding $r$ to the row set of the complementary minor immanant described by $\mathcal{D}_{\Delta}^S$, since by Definition~\ref{dfn:comp-minor-diagram}, white dots in the right column correspond to the row set. On the other hand, if $r \not\in S$, the corresponding gray block is specialized to
    \[\begin{tikzpicture}[baseline,x=.75cm, y=.5cm]
    \begin{scope}[shift={(0,.6)}]
    \draw[circle, fill=white] (0, 0) circle[radius = 1mm] node (D1) {};
    \draw[circle, fill=black] (0, -1) circle[radius = 1mm] node (D2) {};
    \node[draw, dotted, fit=(D1) (D2), inner sep=0mm] {};
    \end{scope}
    \end{tikzpicture}.\]
This corresponds to adding $r-1$ to the row set of the complementary minor immanant described by $\mathcal{D}_{\Delta}^S$. Thus $\mathcal{D}_{\Delta}^S$ has a white dot in the right column for each $r \in S$ and for each $r \in (A(\Delta) \setminus S)^-$, exactly as in $\mathcal{C}_{S \cup (A(\Delta) \setminus S)^-, \,\,\, C(\Delta)}$. For both sides of the equation, all of the remaining dots correspond to the set $(I \setminus S) \cup S^-$, and are all colored black by default. This completes the proof of Equation~\ref{eq:comp-is-dcmd}, and therefore of the theorem.
\end{proof}

With all of the necessary setup completed, we can now offer a proof of Proposition~\ref{lem:main}.

\begin{proof}[Proof of Proposition~\ref{lem:main}]
Let $\Delta$ be any submatrix of $M'$. By Lemma~\ref{lem:dcmd-eq-det}, it suffices to show $\mathcal{D}_{\Delta}\left(\tilde{\Delta}\right) \geq 0$. To start, apply Definition~\ref{dfn:sdcmd} to $\mathcal{D}_{\Delta}$ to obtain
    \[\mathcal{D}_{\Delta}\left(\tilde{\Delta}\right) = \sum_{S \subseteq \mathrm{GB}(\mathcal{D}_{\Delta})} (-1)^{|S|} \mathcal{D}_{\Delta}^S\left(\tilde{\Delta}\right).\]
Then apply Lemma~\ref{lem:tl-compat} to each $\mathcal{D}_{\Delta}^S$ to obtain
    \begin{align}
    \sum_{S \subseteq \mathrm{GB}(\mathcal{D}_{\Delta})} (-1)^{|S|} \mathcal{D}_{\Delta}^S\left(\tilde{\Delta}\right) = \sum_{S \subseteq \mathrm{GB}(\mathcal{D}_{\Delta})} (-1)^{|S|} \left(\sum_{T \in \Theta(S)} \mathrm{imm}_T\left(\tilde{\Delta}\right) \right),
    \label{eq:tl-main-expansion}
    \end{align}
where $\Theta(S)$ is the set of all noncrossing matchings $T \in \mathcal{B}_n$ which are compatible with the complementary minor diagram $\mathcal{D}_{\Delta}^S$; see Definition~\ref{dfn:tl-compat}. Let $\star$ denote the right hand side of Equation~\ref{eq:tl-main-expansion}.

Since $\tilde{\Delta}$ is a generalized submatrix of $M$, it is totally nonnegative. By Lemma~\ref{lem:tl-pos}, each of the $\mathrm{imm}_T\left(\tilde{\Delta}\right)$ is nonnegative. Thus, to finish the proof, we must show that any term with a negative coefficient in $\star$ either cancels with a positive term or evaluates to $0$ on $\tilde{\Delta}$. Recall that $S$ is the set of gray blocks in $\mathcal{D}_{\Delta}$ which are specialized to 
    \begin{tikzpicture}[baseline,x=.75cm, y=.5cm]
    \begin{scope}[shift={(0,.6)}]
    \draw[circle, fill=black] (0, 0) circle[radius = 1mm] node (D1) {};
    \draw[circle, fill=white] (0, -1) circle[radius = 1mm] node (D2) {};
    \node[draw, dotted, fit=(D1) (D2), inner sep=0mm] {};
    \end{scope}
    \end{tikzpicture}
in $\mathcal{D}_{\Delta}^S$. Since every negative term in $\star$ has $|S|$ odd, in particular we have $|S| \geq 1$. Accordingly, we proceed as follows: for each complementary minor diagram with
    \begin{tikzpicture}[baseline,x=.75cm, y=.5cm]
    \begin{scope}[shift={(0,.6)}]
    \draw[circle, fill=black] (0, 0) circle[radius = 1mm] node (D1) {};
    \draw[circle, fill=white] (0, -1) circle[radius = 1mm] node (D2) {};
    \node[draw, dotted, fit=(D1) (D2), inner sep=0mm] {};
    \end{scope}
    \end{tikzpicture}, 
and each noncrossing matching compatible with that diagram, we address the corresponding Temperley-Lieb immanant. We divide these terms into four cases, according to the destination of the strand connected to the white dot in the topmost instance of \begin{tikzpicture}[baseline,x=.75cm, y=.5cm]
    \begin{scope}[shift={(0,.6)}]
    \draw[circle, fill=black] (0, 0) circle[radius = 1mm] node (D1) {};
    \draw[circle, fill=white] (0, -1) circle[radius = 1mm] node (D2) {};
    \node[draw, dotted, fit=(D1) (D2), inner sep=0mm] {};
    \end{scope}
    \end{tikzpicture}.

\textbf{Case I:} There is $T \in \Theta(S)$ such that
    \[\begin{tikzpicture}[baseline, x=.75cm, y=.5cm]
    \begin{scope}[shift={(0,.6)}]
    \draw (0,0) to [bend right=90] (0,-1);
    \draw[circle, fill=black] (0, 0) circle[radius = 1mm] node (D1) {};
    \draw[circle, fill=white] (0, -1) circle[radius = 1mm] node (D2) {};
    \node[draw, dotted, fit=(D1) (D2), inner sep=0mm] {};
    \end{scope}
    \end{tikzpicture}\]
appears as a subdiagram of $T$. In this case, we claim the term cancels with a term of the opposite sign. This is because $T$ is also compatible with the coloring wherein we swap the colors in this block like so
    \[\begin{tikzpicture}[baseline, x=.75cm, y=.5cm]
    \begin{scope}[shift={(0,.6)}]
    \draw (0,0) to [bend right=90] (0,-1);
    \draw[circle, fill=white] (0, 0) circle[radius = 1mm] node (D1) {};
    \draw[circle, fill=black] (0, -1) circle[radius = 1mm] node (D2) {};
    \node[draw, dotted, fit=(D1) (D2), inner sep=0mm] {};
    \end{scope}
    \end{tikzpicture}\]
and leave the rest the same. Since the resulting coloring has one fewer gray block specialized to have black on top, it comes from some set $S'$ with cardinality one less than our original set. Thus $\mathrm{imm}_T\left(\tilde{\Delta}\right)$ appears at least twice in $\star$, with opposite signs. We can remove both terms from $\star$. Note that, since we are considering the topmost instance of this subdiagram, this color-swapping map is well-defined, and in fact a sign-reversing involution on the set of all terms in $\star$ containing at least one of the above two subdiagrams.

\textbf{Case II:} There is $T \in \Theta(S)$ such that
    \[\begin{tikzpicture}[baseline, x=.75cm, y=.5cm]
    \begin{scope}[shift={(0,.6)}]
    \draw (0,-1) to [bend right=60] (0,-3);
    \draw[circle, fill=black] (0, -3) circle[radius = 1mm] node (D4) {};
    \node (D3) at (0, -2) {$\vdots$};
    \draw[circle, fill=black] (0, 0) circle[radius = 1mm] node (D1) {};
    \draw[circle, fill=white] (0, -1) circle[radius = 1mm] node (D2) {};
    \node[draw, dotted, fit=(D1) (D2), inner sep=0mm] {};
    \end{scope}
    \end{tikzpicture}\]
appears as a subdiagram of $T$. In this case, we claim $\mathrm{imm}_T\left(\tilde{\Delta}\right) = 0.$ This is because $T$ is also compatible with the matching in which we swap the colors as follows:
    \[\begin{tikzpicture}[baseline, x=.75cm, y=.5cm]
    \begin{scope}[shift={(0,.6)}]
    \draw (0,-1) to [bend right=60] (0,-3);
    \draw[circle, fill=white] (0, -3) circle[radius = 1mm] node (D4) {};
    \node (D3) at (0, -2) {$\vdots$};
    \draw[circle, fill=black] (0, 0) circle[radius = 1mm] node (D1) {};
    \draw[circle, fill=black] (0, -1) circle[radius = 1mm] node (D2) {};
    \node[draw, dotted, fit=(D1) (D2), inner sep=0mm] {};
    \end{scope}
    \end{tikzpicture}.\]
Now, this coloring does not appear in $\star$, since we always specialize gray blocks to have one black and one white dot; here the block appears with two black dots. But since the total number of black and white dots has not changed, this coloring still corresponds to some valid complimentary minor diagram, which can still be evaluated on $\tilde{\Delta}$. Write the corresponding product of complementary minors as $|\tilde{\Delta}_1| \cdot |\tilde{\Delta}_2|$.

Recall that the white dots in the right column of a complementary minor diagram indicate which rows of $M$ are to be used in $\tilde{\Delta}_1$. By the reasoning in the proof of Lemma~\ref{lem:multilin-expansion-is-tilde-cm}, before the colors were swapped, $\tilde{\Delta}_1$ was upper triangular, with the entries on the main diagonal all either $a_0$ or $b_0$. In particular,
    \begin{tikzpicture}[baseline, x=.75cm, y=.5cm]
    \begin{scope}[shift={(0,.6)}]
    \draw[circle, fill=black] (0, 0) circle[radius = 1mm] node (D1) {};
    \draw[circle, fill=white] (0, -1) circle[radius = 1mm] node (D2) {};
    \node[draw, dotted, fit=(D1) (D2), inner sep=0mm] {};
    \end{scope}
    \end{tikzpicture}
corresponds to placing an $a_0$ on the main diagonal of $\tilde{\Delta}_1$. After swapping colors, the minor has the same column set, since we only swapped colors on the right side. But its row set has been altered by exchanging some row for a lower row. Recall that since $a_0$ is a special SW corner in $M$, all entries below it are $0$. Since we previously had an $a_0$ on the main diagonal and are now using a lower row, the color swap has placed a $0$ on the main diagonal. See the diagram below. Entries boxed in blue belong to $\tilde{\Delta}_1$, while entries boxed in red belong to $\tilde{\Delta}_2$.

    \[\begin{tikzpicture}[baseline]
    \matrix [matrix of math nodes,left delimiter=(,right delimiter=)] (H) at (0,0)
    {b_0 & \ldots & b_{i_1} & b_{i_2} & \ldots & b_{i_\ell} \\
    a_0 & \ldots & a_{i_1} & a_{i_2} & \ldots & a_{i_\ell} \\
    \vdots & & \vdots & \vdots & & \vdots \\
    0 & \ldots & \ast & \ast & \ldots & \ast \\ };
    
    \node[fit=(H-1-6)] (D1) {};
    \node[right=8mm of D1] (d1) {};
    \draw[fill=black] (d1) circle [radius=1mm];

    \node (d2) at (d1 |- H-2-6) {};
    \draw[fill=white] (d2) circle [radius=1mm];

    \node (d3) at (d1 |- H-3-6) {$\vdots$};

    \node (d4) at (d1 |- H-4-6) {};
    \draw[fill=black] (d4) circle [radius=1mm];

    \node[draw, dotted, fit=(d1) (d2), inner sep=1mm] {};

    \node[draw=blue, fit=(H-2-1), inner sep=0pt] {};
    \node[draw=red, fit=(H-1-4) (H-1-6), inner sep=0pt] {};
    \node[draw=red, fit=(H-4-4) (H-4-6), inner sep=0pt] {};
    \end{tikzpicture}
    \quad \rightsquigarrow \quad
    \begin{tikzpicture}[baseline]
    \matrix [matrix of math nodes,left delimiter=(,right delimiter=)] (H) at (0,0)
    {b_0 & \ldots & b_{i_1} & b_{i_2} & \ldots & b_{i_\ell} \\
    a_0 & \ldots & a_{i_1} & a_{i_2} & \ldots & a_{i_\ell} \\
    \vdots & & \vdots & \vdots & & \vdots \\
    0 & \ldots & \ast & \ast & \ldots & \ast \\ };
    
    \node[fit=(H-1-6)] (D1) {};
    \node[right=8mm of D1] (d1) {};
    \draw[fill=black] (d1) circle [radius=1mm];

    \node (d2) at (d1 |- H-2-6) {};
    \draw[fill=black] (d2) circle [radius=1mm];

    \node (d3) at (d1 |- H-3-6) {$\vdots$};

    \node (d4) at (d1 |- H-4-6) {};
    \draw[fill=white] (d4) circle [radius=1mm];

    \node[draw, dotted, fit=(d1) (d2), inner sep=1mm] {};

    \node[draw=blue, fit=(H-4-1), inner sep=0pt] {};
    \node[draw=red, fit=(H-1-4) (H-2-6), inner sep=0pt] {};
    \end{tikzpicture}\]
Thus $|\tilde{\Delta}_1| = 0$, and therefore $|\tilde{\Delta}_1| \cdot |\tilde{\Delta}_2| = 0$. By Lemma~\ref{lem:tl-zero}, any Temperley-Lieb immanant compatible with this coloring is also $0$, so $\mathrm{imm}_T\left(\tilde{\Delta}\right) = 0$.

\textbf{Case III:} There is $T \in \Theta(S)$ such that
    \[\begin{tikzpicture}[baseline, x=.75cm, y=.5cm]
    \begin{scope}[shift={(0,.6)}]
    \draw (0,-1) to (-1.5,-1);
    \draw[circle, fill=black] (0, 0) circle[radius = 1mm] node (D1) {};
    \draw[circle, fill=white] (0, -1) circle[radius = 1mm] node (D2) {};
    \draw[circle, fill=black] (-1.5, -1) circle[radius = 1mm] node {};
    \node[draw, dotted, fit=(D1) (D2), inner sep=0mm] {};
    \end{scope}
    \end{tikzpicture}\]
appears as a subdiagram of $T$. Now, if the black dot in the right column connects upwards to a white dot, we are in case II. So we can assume without loss of generality that it must connect to a white dot in the left column, so that
    \[\begin{tikzpicture}[baseline, x=.75cm, y=.5cm]
    \begin{scope}[shift={(0,.6)}]
    \draw (0,-1) to (-1.5,-1.25);
    \draw (0,0) to (-1.5,.25);
    \node at (-1.5,-.25) {$\vdots$};
    \draw[circle, fill=black] (0, 0) circle[radius = 1mm] node (D1) {};
    \draw[circle, fill=white] (-1.5, 0.25) circle[radius = 1mm] node {};
    \draw[circle, fill=white] (0, -1) circle[radius = 1mm] node (D2) {};
    \draw[circle, fill=black] (-1.5, -1.25) circle[radius = 1mm] node {};
    \node[draw, dotted, fit=(D1) (D2), inner sep=0mm] {};
    \end{scope}
    \end{tikzpicture}\]
appears as a subdiagram. In this case, we claim that $\mathrm{imm}_T\left(\tilde{\Delta}\right)=0$. This is because $T$ is also compatible with the coloring in which we swap the coloring like so
    \[\begin{tikzpicture}[baseline, x=.75cm, y=.5cm]
    \begin{scope}[shift={(0,.6)}]
    \draw (0,-1) to (-1.5,-1.25);
    \draw (0,0) to (-1.5,.25);
    \node at (-1.5,-.25) {$\vdots$};
    \draw[circle, fill=white] (0, 0) circle[radius = 1mm] node (D1) {};
    \draw[circle, fill=black] (-1.5, 0.25) circle[radius = 1mm] node {};
    \draw[circle, fill=black] (0, -1) circle[radius = 1mm] node (D2) {};
    \draw[circle, fill=white] (-1.5, -1.25) circle[radius = 1mm] node {};
    \node[draw, dotted, fit=(D1) (D2), inner sep=0mm] {};
    \end{scope}
    \end{tikzpicture}.\]
As in case II, this coloring does not correspond to any summand in $\star$, since we never change the colors in the left column when constructing $\star$. However, we may still evaluate this complementary minor diagram on $\tilde{\Delta}$ to obtain a complementary minor immanant $|\tilde{\Delta}_1| \cdot |\tilde{\Delta}_2|$.

As in Case II, before the color swap, $|\tilde{\Delta}_1|$ was an minor with $a_0$s and $b_0$s on the main diagonal. Swapping the colors of the dots within the block on the right simply corresponds to changing out an $a_0$ on the main diagonal for a $b_0$.

On the left side, black dots indicate the column set of $\tilde{\Delta}_1$. Thus the color swap corresponds to replacing some column in $\tilde{\Delta}_1$ for a column weakly to the left in $M$. The dots on the left indicate a multiset, so we may have swapped the colors for a pair of duplicate columns. By construction (see Definition~\ref{dfn:d-delta}), however, the second copy of a given column always come second, in white. Thus the original pair could not have been a pair of duplicated columns, so we have in fact moved one of the columns used by $\tilde{\Delta}_1$ strictly to the left. But all entries below and to the left of an $a_0$ and $b_0$ in $M$ are equal to $0$. Thus this new minor is still upper triangular, but now with zeroes on its main diagonal. See the diagram below.

    \[\begin{tikzpicture}[baseline]
    \matrix [matrix of math nodes,left delimiter=(,right delimiter=), nodes in empty cells] (H) at (0,0)
    {a_0 & & & & & & & & \\
    & b_0 & & & & & & & \\
    & & b_0 & & & & & & \\
    & & & a_0 & & & & & \\
    & & & & a_0 & & & & \\
    & & & & & a_0 & & & \\ 
    & & & & & & a_0 & & \\
    & & & & & & & b_0 & \\ 
    & & & & & & & & a_0 \\ };
    \node[draw=black, fit=(H-1-5) (H-9-5), inner sep=2pt] {};
    \end{tikzpicture}
    \quad \rightsquigarrow \quad
    \begin{tikzpicture}[baseline]
    \matrix [matrix of math nodes,left delimiter=(,right delimiter=), nodes in empty cells] (H) at (0,0)
    {a_0 & & & & & & & & \\
    & b_0 & & & & & & & \\
    & & 0 & b_0 & & & & & \\
    & & & 0 & a_0 & & & & \\
    & & & & 0 & & & & \\
    & & & & & a_0 & & & \\ 
    & & & & & & a_0 & & \\
    & & & \leftarrow & & & & b_0 & \\ 
    & & & & & & & & a_0 \\ };
    \node[draw=black, fit=(H-1-3) (H-9-3), inner sep=2pt] {};
    \node[draw, dashed, fit=(H-1-5) (H-9-5), inner sep=3pt] {};
    \end{tikzpicture}\]
It follows that $|\tilde{\Delta}_1| = 0$, and therefore $|\tilde{\Delta}_1| \cdot |\tilde{\Delta}_2| = 0$. Again, by Lemma~\ref{lem:tl-zero}, any Temperley-Lieb immanant compatible with this coloring is $0$. Thus $\mathrm{imm}_T\left(\tilde{\Delta}\right) = 0$.

\textbf{Case IV:} There is $T \in \Theta(S)$ such that
    \[\begin{tikzpicture}[baseline, x=.75cm, y=.5cm]
    \begin{scope}[shift={(0,.6)}]
    \draw (0,2) to [bend right=60] (0,-1);
    \draw[circle, fill=black] (0, 2) circle[radius = 1mm] node (D4) {};
    \node (D3) at (0, 1.3) {$\vdots$};
    \draw[circle, fill=black] (0, 0) circle[radius = 1mm] node (D1) {};
    \draw[circle, fill=white] (0, -1) circle[radius = 1mm] node (D2) {};
    \node[draw, dotted, fit=(D1) (D2), inner sep=0mm] {};
    \end{scope}
    \end{tikzpicture}\]
appears as a subdiagram of $T$. Observe that the topmost black dot in this subdiagram must come from a gray block, and in fact must be at the top of its gray block. If not, then there are strictly more black dots than white dots inside the arc. This is because each white dot is added to the diagram in a pair with a black dot, and there is already one black dot inside the arc. Since there would be more black dots inside the arc than white, one could not complete a noncrossing matching. We assumed, however, that the subdiagram involved the topmost instance of \begin{tikzpicture}[baseline, x=.75cm, y=.5cm]
    \begin{scope}[shift={(0,.6)}]
    \draw[circle, fill=black] (0, 0) circle[radius = 1mm] node (D1) {};
    \draw[circle, fill=white] (0, -1) circle[radius = 1mm] node (D2) {};
    \node[draw, dotted, fit=(D1) (D2), inner sep=0mm] {};
    \end{scope}
    \end{tikzpicture}, so we do not need to consider this case.
\end{proof}

\section{Interlacing Polynomials}\label{sec:interlacing}

We now discuss an application of Proposition~\ref{lem:main} to the theory of interlacing polynomials.

\begin{dfn}\label{dfn:interlacing}
Given a pair of polynomials
\begin{align*}
p_0(t) = a_0 + a_1t + a_2t^2 + \ldots &\text{ with real roots } \chi_1 > \chi_2 > \ldots, \\
p_1(t) = b_0 + b_1t + b_2t^2 + \ldots &\text{ with real roots } \psi_1 > \psi_2 > \ldots,
\end{align*}
we say $p_0$ (weakly) \textit{interlaces} $p_1$ if either $\deg(p_0)=\deg(p_1)$ and the roots satisfy
\[\chi_1 \geq \psi_1 \geq \chi_2 \geq \psi_2 \geq \ldots \geq \chi_n \geq \psi_n,\]
or $\deg(p_0) = \deg(p_1)+1$ and the roots satisfy
\[\chi_1 \geq \psi_1 \geq \chi_2 \geq \psi_2 \geq \ldots \geq \chi_n \geq \psi_n \geq \chi_{n+1}.\]
Note that this definition is assymetrical. The polynomial $p_0$ is assumed to have the largest root overall, and weakly larger degree. By convention, any real-rooted polynomial interlaces the $0$ polynomial.
\end{dfn}

\begin{dfn}\label{dfn:hurwitz}
The \textit{(infinite) Hurwitz matrix} $H(p_0,p_1)$ of $p_0$ and $p_1$ is defined as follows:
\begin{align*}
H\left(p_0, p_1\right)_{i,j} &:=
\begin{cases}
b_{j-1-\lceil \frac{i}{2} \rceil}, & i \text{ odd} \\
a_{j-1-\frac{i}{2}}, & i \text{ even} \\
\end{cases},
\end{align*}
so that
\begin{align*}
H\left(p_0, p_1\right) &= 
\begin{bmatrix}
b_0 & b_1 & b_2 & b_3 & b_4 & \ldots & b_n & 0 & 0 & \ldots \\
a_0 & a_1 & a_2 & a_3 & a_4 & \ldots & a_n & a_{n+1} & 0 & \ldots \\
0 & b_0 & b_1 & b_2 & b_3 & \ldots & b_{n-1} & b_n & 0 & \ldots \\
0 & a_0 & a_1 & a_2 & a_3 & \ldots & a_{n-1} & a_n & a_{n+1} & \ldots \\
0 & 0 & b_0 & b_1 & b_2 & \ldots & b_{n-2} & b_{n-1} & b_n & \ldots \\
0 & 0 & a_0 & a_1 & a_2 & \ldots & a_{n-2} & a_{n-1} & a_n & \ldots \\
0 & 0 & 0 & b_0 & b_1 & \ldots & b_{n-3} & b_{n-2} & b_{n-1} & \ldots \\
0 & 0 & 0 & a_0 & a_1 & \ldots & a_{n-3} & a_{n-2} & a_{n-1} & \ldots \\
\vdots & \vdots & \vdots & \vdots & \vdots & \vdots & \vdots & \vdots & \vdots & \ddots
\end{bmatrix}.
\end{align*}
Note that $H(p_0, p_1)$ is the unfolding of $\begin{bmatrix} p_1(t) \\ p_0(t) \end{bmatrix}$.
\end{dfn}

The following theorem characterizes the deep relationship between interlacing and the Hurwitz matrix:
\begin{thm}[\cite{holtz-tyaglov}, Thm 3.44]\label{thm:known-interlacing}
If $p_0$ and $p_1$ have nonnegative coefficients, then $p_0$ interlaces $p_1$ if and only if $H\left(p_0, p_1 \right)$ is totally nonnegative.
\end{thm}

Using our Proposition~\ref{lem:main}, we can provide an alternative proof of Theorem~\ref{thm:known-interlacing} in the backward direction; i.e. that the total nonnegativity of the Hurwitz matrix forces the polynomials to interlace. To do so, we require the following well-known lemma, which can be viewed as a version of the Routh-Hurwitz~Theorem~\cite{holtz}:
\begin{lem}[\cite{fisk}]\label{lem:routh-hurwitz}
Suppose $p_0$ and $p_1$ are polynomials with nonnegative coefficients. Then $p_0$ interlaces $p_1$ if and only if $p_1(0) \neq 0$ and $p_1(t)$ interlaces $\left(p_0(t)-\frac{p_0(0)}{p_1(0)}p_1(t)\right)/t$.
\end{lem}

\begin{proof}[Proof of Theorem~\ref{thm:known-interlacing}, backward direction]
Suppose $H(p_0, p_1)$ is totally nonnegative. We proceed to show that $p_0$ must interlace $p_1$ by induction on $\deg(p_0)+\deg(p_1)$. The base case occurs when $p_0$ and $p_1$ are both constants. For any pair of constants $a_0, b_0 \geq 0$, the matrix $H(a_0, b_0)$ is clearly totally nonnegative. Constant functions vacuously interlace, since they have no zeroes. Thus the result holds in this case.

Now assume $\deg(p_0)+\deg(p_1) > 0$. Note that $\deg(p_1) \leq \deg(p_0) \leq \deg(p_1)+1$, otherwise we can find a negative $2 \times 2$ minor in $H(p_0, p_1)$. We may also assume without loss of generality that $p_1(0) \neq 0$ and $p_0(0) \neq 0$, which we address in three cases:

\textbf{Case I:} if $p_1(0)=b_0=0$ and $p_0(0) = a_0 \neq 0$, then $H(p_0,p_1)$ has a negative $2 \times 2$ minor unless $p_1$ is identically $0$. Since the $0$ polynomial interlaces any other real-rooted polynomial by convention, we need not consider this case. 

\textbf{Case II:} if $p_1(0) = 0$ and $p_0(0)=0$, then both $p_0$ and $p_1$ are divisible by $t$. From there, it is easy to see that $H\left(p_0, p_1 \right)$ is totally nonnegative if and only if $H\left(p_0 /t , p_1/t\right)$ is totally nonnegative. We may iteratively apply this fact until $p_1(0) \neq 0$. Observe also that $p_0$ interlaces $p_1$ if and only if $tp_0$ interlaces $tp_1$; thus if we assume $p_1(0) \neq 0$, we can return to the general case by multiplying by a power of $t$.

\textbf{Case III:} if $p_0(0) = b_0 \neq 0$ and $p_0(0)=a_0=0$, we may remove the first row and column of $H(p_0, p_1)$ and obtain the the totally nonnegative submatrix $H(p_1,p_0/t)$. Note that, since $p_0$ has $0$ as a root, we know $p_0$ interlaces $p_1$ if and only if $p_1$ interlaces $p_0/t$. Thus we may substitute $p_0 \to p_1$ and $p_1 \to p_0/t$ to continue the proof with two polynomials with nonzero constant term.

Since $H(p_0, p_1)$ is totally nonnegative and $(2,1)$-periodic, we may apply Lemma~\ref{lem:SW-corner-existence} to conclude that it has a special SW corner. Since $p_1(0)\neq 0$ and $p_0(0) \neq 0$, that special SW corner must be in entry $(2,1)$. To eliminate this corner, we will subtract $c=\frac{p_0(0)}{p_1(0)}$ from the second row. We can then apply Proposition~\ref{lem:main} to conclude that the resulting matrix is TN. Observe that, by deleting the first row and column of the eliminated matrix, we obtain $\displaystyle H' = H\left(p_1, \,\,\,\left(p_0-cp_1\right)/t\right)$.

Now notice that $H'$ is a totally positive Hurwitz matrix for a pair of polynomials with total degree
    \[\deg(p_1) + \deg\left(\left(p_0-cp_1\right)/t\right) = \deg(p_1) + \deg(p_0)-1.\]
Thus, by induction, $p_1$ and $\left(p_0-cp_1\right)/t$ weakly interlace. By Lemma~\ref{lem:routh-hurwitz}, it follows that $p_0$ and $p_1$ weakly interlace.
\end{proof}

Theorem~\ref{thm:known-interlacing} can be seen as a necessary and sufficient condition for a $(2,1)$-perdiodic matrix to be TN. For $(1,2)$-periodic matrices, one has the following version of Theorem~\ref{thm:known-interlacing}:

\setcounter{MaxMatrixCols}{30}

\begin{cor}\label{cor:1by2-loops}
Suppose the polynomials
\[p_0(t) = a_0 + a_1t + a_2t^2 + \ldots \text{ and } p_1(t) = b_0 + b_1t + b_2t^2 + \ldots\]
have nonnegative coefficients. Then the unfolding of $\begin{bmatrix} p_1 & p_0 \end{bmatrix}$, i.e. the matrix
\[M = \left[\begin{NiceMatrix}
b_0 & a_0 & b_1 & a_1 & b_2 & a_2 & \ldots & b_n & a_n & b_{n+1} & \ldots \\
0 & 0 & b_0 & a_0 & b_1 & a_1 & \ldots & b_{n-1} & a_{n-1} & b_n & \ldots \\
0 & 0 & 0 & 0 & b_0 & a_0 & \ldots & b_{n-2} & a_{n-2} & b_{n-1} & \ldots \\
\vdots & \vdots & \vdots & \vdots & \vdots & \vdots & \vdots & \vdots & \vdots & \vdots & \ddots \\
\end{NiceMatrix}\right],\]
is totally nonnegative if and only if $t^{\deg(p_1)}p_0(1/t)$ interlaces $t^{\deg(p_1)}p_1(1/t)$.
\end{cor}

\begin{proof}
Let $M$ be as in the statement of the corollary, then let $\widetilde{p_0}(t) = t^{\deg(p_1)}p_0(1/t)$ and let $\tilde{p}_1(t) = t^{\deg(p_1)}p_1(1/t)$. Note that we are mutiplying by $t^{\deg(p_1)}$ in both expressions. The polynomial $\tilde{p}_1(t)$ is precisely $p_1(t)$, but with with the coefficients in reverse order, i.e. with the coefficient of $t^i$ swapped with the coefficient of $t^{\deg(p_1)-i}$. The expression $\widetilde{p_0}(t)$ also has its coefficients in reverse order compared to $p_0(t)$, but possibly with a degree shift.

We claim that $M^\top$ has $H\left(\tilde{p}_0, \tilde{p}_1\right)$ as a submatrix and vice versa. Since a matrix is TN if and only if its transpose is TN, the result will follow. To see that the claim is true, first observe that $\deg(p_0) \leq \deg(p_1) \leq \deg(p_0)+1$; otherwise we can find a negative $2 \times 2$ minor in $M$.

In the first case, when $\deg(p_0) = \deg(p_1)=n$, consider the $(2,1)$-periodic matrix obtained by deleting the first $2n$ columns of $M$, then transposing:
\[\begin{bmatrix}
b_n & b_{n-1} & b_{n-2} & \ldots & b_2 & b_1 & b_0 & 0 & \ldots \\
a_n & a_{n-1} & a_{n-2} & \ldots & a_2 & a_1 & a_0 & 0 & \ldots \\
0 & b_n & b_{n-1} & \ldots & b_3 & b_2 & b_1 & b_0 & \ldots \\
0 & a_n & a_{n-1} & \ldots & a_3 & a_2 & a_1 & a_0 & \ldots \\
\vdots & \vdots & \vdots & \vdots & \vdots & \vdots & \vdots & \vdots & \ddots
\end{bmatrix}.\]
By inspection, the numbers in the first row are the coefficients of $\tilde{p}_1$, and the numbers in the second row are the coefficients of $\tilde{p}_0$. Thus this is $H\left(\tilde{p}_0, \tilde{p}_1\right)$. By deleting the first $n$ columns of $H\left(\tilde{p}_0, \tilde{p}_1\right)$, we recover $M^\top$.

In the second case, when $\deg(p_1) = \deg(p_0)+1 = n+1$, we instead consider the $(2,1)$-periodic matrix obtained by deleting the first $2n+2$ columns of $M$, then transposing:
\[\begin{bmatrix}
b_{n+1} & b_n & b_{n-1} & \ldots & b_2 & b_1 & b_0 & 0 & \ldots \\
0 & a_n & a_{n-1} & \ldots & a_2 & a_1 & a_0 & 0 & \ldots \\
0 & b_{n+1} & b_n & \ldots & b_3 & b_2 & b_1 & b_0 & \ldots \\
0 & 0 & a_n & \ldots & a_3 & a_2 & a_1 & a_0 & \ldots \\
\vdots & \vdots & \vdots & \vdots & \vdots & \vdots & \vdots & \vdots & \ddots
\end{bmatrix}.\]
Again, the numbers in the first row are the coefficients of $\tilde{p}_1$. The numbers in the second row are also still the coefficients of $\tilde{p}_0$; note that they have been shifted up by $1$ because $\deg(p_1) = \deg(p_0)+1$. As before, we conclude this is $H\left(\tilde{p}_0, \tilde{p}_1\right)$. By deleting the first $n+1$ columns of $H\left(\tilde{p}_0, \tilde{p}_1\right)$, we recover $M^\top$.
\end{proof}

By applying either Theorem~\ref{thm:known-interlacing} or Corollary~\ref{cor:1by2-loops} to every $2 \times 1$ or $1 \times 2$ submatrix of a general loop, we obtain the following corollary:

\begin{cor}\label{cor:same-row-col}
Suppose that
\[M = \begin{bmatrix} p_{1,1}(t) & p_{1,2}(t) & \ldots & p_{1,m}(t) \\
p_{2,1}(t) & p_{2,2}(t) & \ldots & p_{2,m}(t) \\
\vdots & \vdots & \ddots & \vdots \\
p_{n,1}(t) & p_{n,2}(t) & \ldots & p_{n,m}(t) \end{bmatrix}\] 
has a TN unfolding. Then for all $i < i'$ and $j < j'$, 
{\setstretch{1.5}
\begin{itemize}
\item the polynomial $p_{i',j}(t)$ interlaces $p_{i,j}(t)$ and 
\item the polynomial $t^{\deg(p_{i,j})}p_{i,j'}(1/t)$ interlaces $t^{\deg(p_{i,j})}p_{i,j}(1/t)$.
\end{itemize}
}
\end{cor}

The converse of Corollary~\ref{cor:same-row-col} is not true. Consider the matrix of polynomials below and its associated unfolding:
\begin{align*}
{\renewcommand*{\arraystretch}{2.3}
\left[\begin{array}{cc}
(t+3)(t+7) & (t+4)(t+10) \\
(t+2)(t+5) & (t+3)(t+6) \\
\end{array} \right]}
& \underset{\text{Unfold}}{\leadsto}
\left[\begin{array}{cc|cc|cc|cc|c}
21 & 40 & 10 & 14 & 1 & 1 & 0 & 0 & \ldots \\
10 & 18 & 7 & 9 & 1 & 1 & 0 & 0 & \ldots \\
\hline
0 & 0 & 21 & 40 & 10 & 14 & 1 & 1 & \ldots \\
0 & 0 & 10 & 18 & 7 & 9 & 1 & 1 & \ldots \\
\hline
\vdots & \vdots & \vdots & \vdots & \vdots & \vdots & \vdots & \vdots & \ddots
\end{array} \right]
\end{align*}
Although this matrix satisfies the conclusions of Corollary~\ref{cor:same-row-col}, it is not TN, since it has the following minor:
\[\left|\begin{matrix} 21 & 40 \\ 10 & 18 \end{matrix}\right| = -22.\]
Thus prompting the following open question:
\begin{qstn}
Suppose $M$ is as in Corollary~\ref{cor:same-row-col}. Can we find necessary and sufficient conditions on the polynomials $p_{i,j}$?
\end{qstn}



\section*{Acknowledgements}

I would like to thank my advisor, Pasha Pylyavskyy, for his guidance and support while working on this project; professors Christine Berkesch, Gregg Musiker, and Victor Reiner for reading early drafts of this work; fellow graduate students Joe McDonough, Anastasia Nathanson, Lilly Webster, and Sylvester Zhang for hearing out my ideas and helping me test examples. I received partial support from NSF grant DMS-1949896.
\printbibliography

\end{document}